\documentclass[a4paper, 10pt]{amsart}

\usepackage{amsthm, amsmath, amssymb, amsfonts}
\usepackage{xcolor}
\usepackage{enumitem}

\usepackage[colorlinks, linkcolor=cyan, citecolor=cyan, urlcolor=blue, bookmarks=false, hypertexnames=true]{hyperref}

\theoremstyle{plain}
\newtheorem{theorem}{Theorem}[section]
\newtheorem*{theorem*}{Theorem \ref{thm:appl}}
\newtheorem{proposition}[theorem]{Proposition}
\newtheorem{conjecture}{Conjecture}

\newtheorem{lemma}[theorem]{Lemma}

\theoremstyle{definition}

\newtheorem{remark}[theorem]{Remark}

\numberwithin{equation}{section}

\DeclareMathOperator{\trace}{trace}
\DeclareMathOperator{\grad}{grad}

\DeclareMathOperator{\Div}{div}
\DeclareMathOperator{\id}{Id}

\DeclareMathOperator{\End}{End}

\newcommand{\uu}{\mathcal{U}}
\newcommand{\ww}{\mathcal{W}}
\newcommand{\xx}{\mathcal{X}}
\newcommand{\yy}{\mathcal{Y}}
\newcommand{\kk}{\mathcal{K}}
\newcommand{\vv}{\mathcal{V}}
\newcommand{\zz}{\mathcal{Z}}
\newcommand{\U}{\mathcal{Q}}
\newcommand{\F}{\mathcal{F}}

\allowdisplaybreaks

\title[Biconservative surfaces]{Biconservative Weingarten surfaces with flat normal bundle in $N^4 (\epsilon)$}

\author{\c Stefan Andronic, Stefano Montaldo, Cezar Oniciuc and Antonio Sanna}

\address{Faculty of Mathematics, Al. I. Cuza University of Iasi, Blvd. Carol I, no. 11, 700506 Iasi, Romania} \email{stefanandronic215@gmail.com}

\address{Dipartimento di Matematica e Informatica, Universit\` a degli Studi di Cagliari, Via Ospedale 72, 09124 Cagliari, Italia} \email{montaldo@unica.it} 

\address{Faculty of Mathematics, Al. I. Cuza University of Iasi, Blvd. Carol I, no. 11, 700506 Iasi, Romania} \email{oniciucc@uaic.ro}

\address{Address: Dipartimento di Matematica e Informatica, Universit\` a degli Studi di Cagliari, Via Ospedale 72, 09124 Cagliari, Italia}

\address{Current Address: Faculty of Mathematics, Al. I. Cuza University of Iasi, Blvd. Carol I, no. 11, 700506 Iasi, Romania}

\email{antonio.sanna4@unica.it}


\thanks{The authors S.M. and A.S. are members of the Italian National Group G.N.S.A.G.A. of INdAM. The author A. S. was supported by a NRRP scholarship - funded by the European Union - NextGenerationEU - Mission 4, Component 1, Investment 3.4.
}

\begin{document}
	\maketitle
	
	\begin{abstract}
		In this paper, we extend our investigation of the class of biconservative surfaces with non-constant mean curvature in 4-dimensional space forms $N^4(\epsilon)$. Specifically, we focus on biconservative surfaces with non-parallel normalized mean curvature vector fields (non-PNMC) that have flat normal bundles and are Weingarten. In our initial result we obtain the compatibility conditions for this class of biconservative surfaces in terms of an ODE system. Subsequently, by prescribing the flat connection in the normal bundle, we prove an existence result for the considered class of biconservative surfaces. Furthermore, we determine all non-PNMC biconservative Weingarten surfaces with flat normal bundles that either exhibit a particular form of the shape operator in the direction of the mean curvature vector field or have constant Gaussian curvature $K = \epsilon$. Finally, we prove that such surfaces cannot be biharmonic.	
	\end{abstract}
	
	\section{Introduction}
	
	In recent years, the theory of biconservative submanifolds has undergone substantial development as an effort to generalize the notion of biharmonic submanifolds. The biharmonic isometric immersions $\varphi : M^m \to N^n$, that is biharmonic submanifolds, are characterized by the vanishing of the bitension field
	\[
	\frac 1 m \tau_2 (\varphi) = - \Delta^\varphi H - \trace R^N \left ( d \varphi (\cdot), H \right ) d \varphi (\cdot) = 0,
	\]
	where $\Delta ^\varphi$ is the rough Laplacian acting on sections of the pull-back bundle $\varphi ^{-1} \left ( TN ^n \right )$, $R^N$ denotes the curvature tensor field on $N^n$ and $H$ is the mean curvature vector field associated to the immersion $\varphi$. Of course, any minimal submanifold, that is $H = 0$, is biharmonic and we are interested in studying biharmonic submanifolds which are non-minimal, called proper biharmonic. Naturally, the biharmonic equation decomposes into its tangent and normal components. 
	
	The study of biharmonic submanifolds, due to frequent incompatibility between the normal and tangent components, has proved to be relatively rigid. To overcome this rigidity, the biconservative submanifolds are defined by the vanishing of the tangent component of the biharmonic equation, that is
	\[
	\left ( \tau_2 (\varphi) \right )^\top = 0.
	\]
	Biconservative submanifolds can be also characterized (in fact, this was the original definition, see \cite{CaddeoMontaldoOniciucPiu2014}) as the submanifolds with divergence-free stress-bienergy tensor $S_2$, where $S_2$ has a variational meaning (see \cite{Jiang1987} and \cite{LoubeauMontaldoOniciuc2008}). For recent surveys on this topic we refer to \cite{FetcuOniciuc2022} and \cite{FuZhan2023}.
	
	The study of biconservative submanifolds started in 1995 when Hasanis and Vlachos, in an attempt to solve the Chen Conjecture (see \cite{Chen1991}) in dimension $4$, classified all biconservative hypersurfaces in the $4$-dimensional Euclidean space $\mathbb R^4$, see \cite{HasanisVlachos1995}. In their paper, the biconservative hypersurfaces in Euclidean spaces were referred to as H-hypersurfaces. Then, there were studied the biconservative hypersurfaces in other space forms (see, for example, \cite{BibiSoretVille2024Arxiv}, \cite{FuHongYangZhan2025} and \cite{MontaldoOniciucRatto2016Annali}). An important feature of biconservative hypersurfaces in space forms is that those with constant mean curvature (CMC) are inherently biconservative, while in the non-CMC case, one of its principal directions is spanned by the gradient of the mean curvature function and the corresponding principal curvature is a certain constant multiplied by the mean curvature function. 
	
	In the case of submanifolds with codimension greater than $1$, the study of the biconservativity becomes more challenging, although the geometry shows more richness.
	
	In the special case of biconservative surfaces, there have been obtained interesting results. For example, for any biconservative surface in an arbitrary target manifold $N^n$, the generalized Hopf function 
	\[
	Q = \langle B (\partial_z, \partial_z), H \rangle
	\]
	is holomorphic if and only if $M^2$ has constant mean curvature (see \cite{LoubeauOniciuc2014}, \cite{MontaldoOniciucRatto2016JGeoAn} and \cite{Nistor2017}). Here, $B$ denotes the second fundamental form of the surfaces and $\partial _z = \left ( \partial _x - i \partial _y \right ) / 2$, where $\left (x, y \right )$ are isothermal coordinates.
	
	In \cite{MontaldoOniciucRatto2016JGeoAn} the rigidity of CMC biconservative surfaces in $4$-dimensional space forms with non-zero sectional curvature was proved. More precisely, such surfaces must have parallel mean curvature tensor field (PMC).
	
	The non-CMC case is difficult to handle and it is necessary to impose additional hypotheses. A natural one is to consider the surfaces with parallel normalized mean curvature vector field (PNMC). All non-CMC, PNMC biconservative surfaces in $4$-dimensional space forms were classified in \cite{NistorOniciucTurgaySen2023}, \cite{NistorRusu2024} and \cite{SenTurgay2018} and these surfaces have two important properties: they have flat normal bundle and they are Weingarten surfaces (W-surfaces).
	
	In our paper, we continue the study of biconservative surfaces in $4$-dimensional space forms in the non-CMC case by relaxing the PNMC hypothesis and, naturally, considering non-CMC (non-PNMC) biconservative W-surfaces with flat normal bundle. We first explore the case of surfaces satisfying the condition $3 f^2 + K - \epsilon \neq 0$, where $f$ is the mean curvature function and $K$ is the Gaussian curvature of the surface. As a first main result, we prove that such surfaces are characterized by a first order ODE system \eqref{eq:differentialSystemTangentPartConverse}. Moreover, this system represents the compatibility condition for this class of biconservative surfaces (see Theorems \ref{th:directTheorem} and \ref{th:existenceTheorem}). This ensures the existence of non-CMC (non-PNMC) biconservative W-surfaces with flat normal bundle. Using system \eqref{eq:differentialSystemTangentPartConverse}, we determine all non-CMC biconservative W-surfaces with flat normal bundle for which one of its principal directions is spanned by the gradient of the mean curvature function and the corresponding eigenvalue for the shape operator in the direction of the mean curvature vector field is twice the mean curvature function (see Theorem \ref{th:exampleYZero}). Then, using a reformulation of the main system \eqref{eq:differentialSystemTangentPartConverse}, that is system \eqref{eq:differentialSystemTangentPartWithK}, we find all non-CMC biconservative W-surfaces with flat normal bundle and constant Gaussian curvature $K = \epsilon$ in any $4$-dimensional space form $N^4 (\epsilon)$ (see Theorem \ref{th:exampleKEqualsEpsilon}). A more general solution of system \eqref{eq:differentialSystemTangentPartWithK} is presented in Proposition \ref{th:exampleMoreGeneral}. In the last part of this section we investigate the case when $3 f^2 + K - \epsilon = 0$. We show in Theorem \ref{th:twoImmersionsKEqualsEps-3F2} that, for a certain given abstract surface, there exists a $1$-parametric family of non-CMC biconservative W-immersions with flat normal bundle in a $4$-dimensional space form with $3 f^2 + K - \epsilon = 0$ but there also exists a unique non-CMC biconservative immersion in a $3$-dimensional space form $N^3 (\epsilon)$ with the same mean curvature $f$.
	
	As a byproduct of our work, we review the case of non-CMC, PNMC biconservative surfaces studied in \cite{NistorOniciucTurgaySen2023}, \cite{NistorRusu2024} and \cite{SenTurgay2018}, as a singular case of our first order ODE system. The PNMC biconservative surfaces are characterized by system \eqref{eq:differentialSystemPNMC}. Using our approach, we reprove two properties of PNMC biconservative surfaces which do not hold in the non-PNMC case: if we fix the domain abstract surface $\left ( M^2, g \right )$, then there exists at most one PNMC biconservative immersion $\varphi : \left ( M^2, g \right ) \to N^4 (\epsilon)$ (see Theorem \ref{th:uniquenessPNMC}); next, we redetermine all abstract surfaces $\left (M^2, g \right )$ which admit (unique) PNMC biconservative immersions (see Proposition \ref{th:existencePNMC}).
	
	In the last part of our paper, we investigate the biharmonicity of non-CMC W-surfaces with flat normal bundle in $4$-dimensional space forms. For this, we extend the first order ODE system \eqref{eq:differentialSystemTangentPartConverse} to the biharmonic case and show that there are neither such surfaces with the shape operator described in Theorem \ref{th:exampleYZero}, nor such surfaces with constant Gaussian curvature given in Theorem \ref{th:exampleKEqualsEpsilon}.
	
	We end the paper with our conviction about the (non-)existence of non-CMC biharmonic W-surfaces with flat normal bundle in $4$-dimensional space forms (see Conjecture \ref{th:OpenProblem}). The validity of Conjecture \ref{th:OpenProblem} is the key point in the proof of the classification result stated in Theorem \ref{th:theoremAfterOpenProblem}.
	
	Our belief concerning the full classification of proper biharmonic surfaces in $4$-dimensional space forms is expressed in Conjecture \ref{th:conjecture}.
	
	
	\noindent \textbf{Conventions.} In this paper, all manifolds are assumed to be connected. Also, all immersions are assumed to be isometric immersions. The metrics on arbitrary manifolds will be denoted by $\langle \cdot, \cdot \rangle$ or will not be explicitly indicated. 
	
	Let $M$ be a Riemannian manifold and denote by $\nabla$ the Levi-Civita connection of $M$. The rough Laplacian acting on the set of all sections in an arbitrary vector bundle $\Upsilon$ over $M$ is given by
	\[
	\Delta ^\Upsilon = - \trace \left ( \nabla ^\Upsilon \nabla ^\Upsilon - \nabla ^\Upsilon _\nabla \right ),
	\]
	where $\nabla ^\Upsilon$ is an affine connection on $\Upsilon$, and the curvature tensor field is 
	\[
	R ^\Upsilon (X, Y) \sigma = \nabla ^\Upsilon _X \nabla ^\Upsilon _Y \sigma - \nabla ^\Upsilon _Y \nabla ^\Upsilon _X \sigma - \nabla ^\Upsilon _{[X, Y]} \sigma,
	\]
	for any $X, Y \in C(TM)$ and any $\sigma \in C(\Upsilon)$. By $\mathbb R^*$ we denote the set of all non-zero numbers.
	
	\section{Preliminaries}
	
	In this section we fix the notations used in this paper and present some known results which will be useful later. 
	
	Let $\varphi : M^m \to N^n (\epsilon)$ be an immersion, that is $M^m$ is a submanifold of $N^n (\epsilon)$. Locally, we can identify $M^m$ with its image through $\varphi$, a tangent vector field $X$ with $d \varphi (X)$ and the connection in the pull-back bundle $\nabla ^\varphi _X d \varphi (Y)$ with $\nabla ^N _X Y$, where $\nabla ^N$ is the Levi-Civita connection on $N^n (\epsilon)$. The Gauss and the Weingarten formulas are
	\[
	\nabla ^N _X Y = \nabla _X Y + B (X, Y), \quad X, Y \in C(TM)
	\]
	and
	\[
	\nabla ^N _X \eta = - A_\eta X + \nabla ^\perp _X \eta, \quad \eta \in C(NM),
	\]
	respectively, where $B \in C \left ( \odot ^2 T^* M \otimes NM \right )$ is called the \textit{second fundamental form} of $M^m$ in $N^n (\epsilon)$, $A_\eta \in C \left ( T^* M \otimes TM \right )$ is the \textit{shape operator} of $M^m$ in $N^n (\epsilon)$ in the direction of $\eta \in C(NM)$ and $\nabla ^\perp$ is the \textit{induced connection} in the \textit{normal bundle} $NM$ of $M^m$ in $N^n (\epsilon)$. The \textit{mean curvature vector field} of $M^m$ in $N^n (\epsilon)$ is 
	\[
	H = \frac 1 m \trace B.
	\] 
	We denote by $R$ the curvature tensor field of $M$.
	
	Now, we recall the fundamental equations of an arbitrary submanifold $M^m$ in a space form $N^n (\epsilon)$. 
	
	The \textit{Gauss equation} is 
	\begin{equation} \label{eq:GaussEquationInSpaceForms}
		\epsilon \left ( \langle X, W \rangle \langle Y, Z \rangle - \langle X, Z \rangle \langle Y, W \rangle \right ) = \langle R (X, Y) Z, W \rangle + \langle B (X, Z), B (Y, W) \rangle - \langle B (X, W), B(Y, Z) \rangle,
	\end{equation}
	for any $X, Y, Z, W \in C(TM)$.
	
	The \textit{Codazzi equation} is 
	\begin{equation} \label{eq:CodazziEquationInSpaceForms}
		\left ( \nabla ^\perp _X B \right ) (Y, Z) = \left ( \nabla ^\perp _Y B \right ) (X, Z),
	\end{equation}
	for any $X, Y, Z \in C(TM)$.
	
	The \textit{Ricci equation} is 
	\begin{equation} \label{eq:RicciEquationInSpaceForms}
		\left \langle R ^\perp (X, Y) \xi, \eta \right \rangle = \left \langle \left [ A_\xi, A_\eta \right ] X, Y \right \rangle,
	\end{equation}
	for any $X, Y \in C(TM)$ and for any $\xi, \eta \in C(NM)$, where $R ^\perp$ is the curvature tensor field in the normal bundle $NM$.
	
	A submanifold $M^m$ of a space form $N^n (\epsilon)$ is said to have \textit{flat normal bundle} if the curvature tensor field in the normal bundle vanishes identically, that is
	\[
	R ^\perp = 0.
	\]
	For geometric properties of submanifolds with flat normal bundle we refer the reader to \cite{DajczerTojeiroBook2019}.
	
	Next, we recall a characterization result for biharmonic submanifolds in space forms
	
	\begin{theorem} [\cite{ChenBook2015}, \cite{Oniciuc2002}]
		Let $\varphi : M^m \to N^n (\epsilon)$ be an immersion. Then, $\varphi$ is biharmonic if and only if
		\begin{equation} \label{eq:biharmonicEquationTangentPart}
			2 \trace A_{\nabla ^\perp _{(\cdot)} H} (\cdot) + \frac m 2 \grad |H|^2 = 0
		\end{equation}
		and
		\begin{equation} \label{eq:biharmonicEquationNormalPart}
			\Delta ^\perp H + \trace B (\cdot, A_H (\cdot) ) - m \epsilon H = 0.
		\end{equation}
	\end{theorem}

	Equations \eqref{eq:biharmonicEquationTangentPart} and \eqref{eq:biharmonicEquationNormalPart} represent the vanishing of the tangent and normal components of the bitension vector field $\tau_2 (\varphi)$, respectively. Consequently, a biconservative immersion is characterized only by \eqref{eq:biharmonicEquationTangentPart}. 
	
	We just recall here that the \textit{stress-bienergy} tensor $S_2$ associated to an immersion $\varphi : M^m \to N^n$ is given by 
	\begin{equation*}
		S_2 = - \frac {m^2} 2 |H|^2 \id + 2 m A_H
	\end{equation*}
	and it satisfies 
	\begin{equation*}
		\left ( \Div S_2 \right )^\# = - \left ( \tau_2 (\varphi) \right )^\top.
	\end{equation*}
	
	Following \cite{Ishikawa1992}, we define \textit{W-surfaces}, or \textit{Weingarten surfaces}, in $4$-dimensional space forms as immersions $\varphi : M^2 \to N^4 (\epsilon)$ such that there exists a smooth function $W : \mathbb R^2 \to \mathbb R$, $W = W (x, y)$ with non-zero gradient everywhere
	such that 
	\[
	W (f, K) = 0, \quad \text{on } M,
	\]
	where $f = |H|$ is the \textit{mean curvature function} of $M$ and $K$ is the \textit{Gaussian curvature} of $M$.
	
	\section{Biconservative surfaces}
	
	Let $\varphi : M^2 \to N^4 (\epsilon)$ be a surface and assume that $H \neq 0$ at any point of $M$. Then, the mean curvature function $f = |H|$ is smooth and we set 
	
	\begin{equation} \label{eq:MeanCurvatureNormalized}
		E_3 = \frac{1}{f} H \in C(NM).
	\end{equation}
	Let $\{ E_1, E_2 \}$ be an orthonormal frame field tangent  to $M$ defined on an open subset $U \subset M$ and let $E_4 \in C(NM)$ be a unit normal section orthogonal to $E_3$. Since our results will be of local nature, we assume that $U = M$.
	
	We can assume that $\{ E_a \}_{a=1}^4$ is the restriction to $M$ of a local orthonormal frame field on $N^4$, also denoted by $\{ E_a \}_{a=1}^4$. Let  $\omega_a^b \in \Lambda^1 \left (N^4 \right )$, $1 \leq a,b \leq 4$ be the \textit{connection forms} on $N^4$ with respect to $E_a$, defined by
	\[
		\nabla_V^N E_a = \omega_a^b(V) E_b, \quad \text{for any } V \in C \left (TN^4 \right ).
	\]
	We use the same notation $\omega_a^b$ for the pull-back $\varphi^* \omega_a^b$ and it will be clear from the context to which of them we are referring to. It can be shown that the $1$-form  $\omega_1^2 \in \Lambda^1(M)$ is the connection form of $M$, that is 
	\[
	\nabla _X E_1 = \omega^2_1 (X) E_2, \quad \text{for any } X \in C \left (TM^2 \right ).
	\] 
	Denote by $A_3$ and $A_4$ the shape operators associated to $E_3$ and $E_4$, respectively. 
	
	In the following we assume that $\grad f \neq 0$ at any point of $M$. This assumption is natural since CMC biconservative surfaces in $4$-dimensional space forms $N^4 (\epsilon)$ were classified in \cite{MontaldoOniciucRatto2016JGeoAn}. It is known that \textit{pseudo-umbilical}, that is $A_3 = f \id$, biconservative surfaces in $N^4 (\epsilon)$ are CMC (see \cite{CaddeoMontaldoOniciuc2002} and \cite{Dimitric1992}). Thus, we further assume that $A_3 \neq f \id$ at any point of $M$. This means that the eigenvalues of $A_3$ have constant multiplicities equal to $1$ at any point. It follows that, the eigenvalue functions $k_1$ and $k_2$ are smooth functions on $M$ and locally we can choose a frame field $\{ E_1, E_2 \}$ tangent to $M$ such that
	\begin{equation} \label{eq:ShapeOperator3}
		A_3 E_1 = k_1 E_1, \quad  A_3 E_2 = k_2 E_2.
	\end{equation}
	We assume that $E_1$ and $E_2$ are defined on $M$ and denote by $\omega^1, \omega^2 \in \Lambda^1(M)$ the dual frame field of $\{ E_1, E_2 \}$ on $M$. With respect to this dual frame we have
	\begin{equation} \label{eq:Omega12}
		\omega_2^1  = - \omega_1^2 = a_1 \omega^1 + a_2 \omega^2, 
	\end{equation}
	and
	\begin{equation} \label{eq:Omega43}
		\omega_3^4  = - \omega_4^3 = b_1 \omega^1 + b_2 \omega^2,
	\end{equation}
	where $a_1, a_2,  b_1, b_2 \in C^{\infty}(M)$.
	
	When the surface is Weingarten and has flat normal bundle, it enjoys several properties which we describe in the following results. We point out that the Propositions \ref{th:propertiesOurSurface}, \ref{th:propertiesBiconservativeSurface}, \ref{th:normalConnection} and Lemma \ref{th:constraintNormalConnection} are similar to those proved in \cite{Ishikawa1992}. We have added the statements and their proofs for completeness and for consistency of the notations.
	
	\begin{proposition} \label{th:propertiesOurSurface}
		Let $\varphi: M^2 \to N^4(\epsilon)$ be a surface with flat normal bundle. Assume that $H \neq 0$, $\grad f \neq 0$ and $A_3 \neq f \id$ at any point. Then, the following hold
		
		\begin{enumerate}[label = \alph*), itemsep=5pt]
			\item \label{item:A4} the shape operator in the direction of $E_4$ is given by
			\begin{equation} \label{eq:ShapeOperator4}
				A_4 E_1 = \alpha E_1 \quad \text{and} \quad A_4 E_2 = -\alpha E_2,
			\end{equation}
			 for some $\alpha \in C^{\infty}(M)$;
			 \item \label{item:secondFundamentalForm} the second fundamental form is given by
			 \begin{equation} \label{eq:SecondFundamentalForm}
			 	B(E_1, E_1) = k_1 E_3 + \alpha E_4, \quad  B(E_1, E_2)=B(E_2, E_1)=0, \quad B(E_2, E_2) = k_2 E_3 - \alpha E_4;
			 \end{equation}
			 \item \label{item:connections} the Levi-Civita connection of $M$ is given by
			 \begin{equation} \label{eq:ConnectionFormSurface}
			 	\nabla_{E_1} E_1 = -a_1 E_2, \quad \nabla_{E_1} E_2 = a_1 E_1, \quad \nabla_{E_2} E_1 = -a_2 E_2, \quad \nabla_{E_2} E_2 = a_2 E_1,
			 \end{equation}
			 the connection in the normal bundle takes the expression
			 \begin{equation} \label{eq:ConnectionFormNormal}
			 	\nabla_{E_1}^\perp E_3 = b_1 E_4, \quad \nabla_{E_1}^\perp E_4 = -b_1 E_3, \quad \nabla_{E_2}^\perp E_3 = b_2 E_4, \quad \nabla_{E_2}^\perp E_4 = -b_2 E_3
			 \end{equation}
		 	 and 
		 	 \begin{equation} \label{eq:FlatNormalBundleEquation}
		 	 	E_1(b_2) - E_2(b_1) = a_1 b_1 + a_2 b_2;
		 	 \end{equation}
			 \item \label{item:gaussianCurvature} the Gaussian curvature is
			 \begin{equation} \label{eq:GaussEquation}
			 	 K = E_1(a_2) - E_2(a_1) - a_1^2 - a_2^2;
			 \end{equation}
			 
			 \item \label{item:fundamentalEquations} the Gauss and Codazzi equations are equivalent to
			 \begin{align} 
			 	& K = \epsilon + k_1 k_2 - \alpha^2 \label{eq:GaussianCurvatureFromEquation} \\
			 	& E_1(k_2) = a_2(k_2 - k_1) - \alpha b_1 \label{eq:E1kappa2} \\
			 	& E_2(k_1) = a_1(k_2 - k_1) + \alpha b_2 \label{eq:E2kappa1} \\
			 	& E_1(\alpha) = 2\alpha a_2 + k_2 b_1 \label{eq:E1alpha} \\
			 	& E_2(\alpha) = -2\alpha a_1 - k_1 b_2; \label{eq:E2alpha}
			 \end{align}
			 
			 \item \label{item:WSurface} moreover, if $M^2$ is also a W-surface, then $\grad f$ is parallel to $\grad K$, that is
			 \begin{equation} \label{eq:WSurfaceEquaition}
			 	\grad f || \grad K.
			\end{equation}	
		\end{enumerate}
	\end{proposition}
	
	\begin{proof}
		Item \ref{item:A4}: First of all, using \eqref{eq:MeanCurvatureNormalized}, we have
		\[
			\trace A_4 = \sum_{i=1}^2 \langle A_4 E_i, E_i \rangle = \sum_{i=1}^2 \langle B ( E_i, E_i), E_4 \rangle = 2 \langle H, E_4  \rangle = 2f \langle E_3, E_4 \rangle,
		\]
		that is
		\[
			\trace A_4 = 0.
		\]
		Next, using the Ricci equation \eqref{eq:RicciEquationInSpaceForms} and the fact that $M$ has flat normal bundle, we obtain
		\[
			\langle [A_3, A_4]X, Y \rangle = 0, \quad \forall X, Y \in C(TM)
		\]
		and this implies that $[A_3, A_4] = 0$, that is $A_3 \circ A_4 = A_4 \circ A_3$. In the following we show that $\{ E_1, E_2 \}$ diagonalizes also $A_4$. Indeed, let $A_4 E_1 = \alpha E_1 + \beta E_2$, where $\alpha, \beta \in C^\infty(M)$. Then, we have 
		\[
			A_3( A_4 E_1) = A_3(\alpha E_1 + \beta E_2) = \alpha A_3 E_1 + \beta A_3 E_2 = \alpha k_1 E_1 + \beta k_2 E_2.
		\]
		On the other hand, 
		\[
			A_3 (A_4 E_1) = A_4 (A_3 E_1) = A_4(k_1 E_1) =  \alpha k_1 E_1 + \beta k_1 E_2.
		\]
		Thus, $\beta(k_2 - k_1) = 0$ must hold. Since $M$ is non pseudo-umbilical, we obtain  $\beta = 0$ on $M$, that is $A_4 E_1 = \alpha E_1$.
		
		Now consider $a, b \in C^\infty(M)$ such that $A_4 E_2 = a E_1 + b E_2$. Using similar computations, we obtain $a(k_2 - k_1) = 0$ and the non pseudo-umbilical condition yields $a = 0$ on $M$, that is $A_4 E_2 = b E_2$. Since $\trace A_4 = 0$, we obtain $A_4 E_2 = - \alpha E_2$, and \eqref{eq:ShapeOperator4} is proved.
		
		Item \ref{item:secondFundamentalForm}: In order to prove \eqref{eq:SecondFundamentalForm}, we combine \eqref{eq:ShapeOperator3} and \eqref{eq:ShapeOperator4} and obtain
		\begin{align*}
			B(E_1, E_1)  = & \langle B(E_1, E_1), E_3\rangle E_3 + \langle B(E_1, E_1), E_4\rangle E_4\\
			 			 = & \langle A_3 E_1, E_1\rangle E_3 + \langle A_4 E_1, E_1\rangle E_4 = k_1 E_3 + \alpha E_4,\\
			B(E_1, E_2)  = & \langle B(E_1, E_2), E_3\rangle E_3 + \langle B(E_1, E_2), E_4\rangle E_4\\
						 = & \langle A_3 E_1, E_2\rangle E_3 + \langle A_4 E_1, E_2\rangle E_4 = 0,\\
			B(E_2, E_2)  = & \langle B(E_2, E_2), E_3\rangle E_3 + \langle B(E_2, E_2), E_4\rangle E_4\\
						 = & \langle A_3 E_2, E_2\rangle E_3 + \langle A_4 E_2, E_2\rangle E_4 = k_2 E_3 - \alpha E_4.
		\end{align*}
		Item \ref{item:connections}: Now, we compute the Levi-Civita connection of $M$. From \eqref{eq:Omega12}, we have
		\begin{align*}
			\nabla_{E_1} E_1 = & \omega_1^1(E_1) E_1 + \omega_1^2(E_1) E_2 = - \left ( a_1\omega^1 + a_2 \omega^2 \right )(E_1)E_2 = -a_1 E_2,\\
			\nabla_{E_1} E_2 = & \omega_2^1(E_1) E_1 + \omega_2^2(E_1) E_2 = \left ( a_1\omega^1 + a_2 \omega^2 \right )(E_1)E_1 = a_1 E_1,\\
			\nabla_{E_2} E_1 = & \omega_1^1(E_2) E_1 + \omega_1^2(E_2) E_2 = - \left ( a_1\omega^1 + a_2 \omega^2 \right )(E_2)E_2 = -a_2 E_2,\\
			\nabla_{E_2} E_2 = & \omega_2^1(E_2) E_1 + \omega_2^2(E_2) E_2 = \left ( a_1\omega^1 + a_2 \omega^2 \right ) (E_2)E_1 = a_2 E_1.		
		\end{align*}
		Next, we compute the connection $\nabla^\perp$ in the normal bundle. Using \eqref{eq:Omega43}, we have
		\begin{align*}
			\nabla_{E_1}^\perp E_3 = & \left \langle \nabla_{E_1}^\perp E_3, E_3 \right \rangle E_3 + \left \langle \nabla_{E_1}^\perp E_3, E_4 \right \rangle E_4\\
			= & \left \langle \nabla^N_{E_1} E_3, E_3 \right \rangle E_3 + \left \langle \nabla^N_{E_1} E_3, E_4 \right \rangle E_4\\
			= & \left \langle \omega_3^3(E_1) E_3 + \omega_3^4(E_1) E_4, E_3 \right \rangle E_3 + \left \langle \omega_3^3(E_1) E_3 + \omega_3^4(E_1) E_4, E_4 \right \rangle E_4\\
			= & \omega_3^4(E_1) E_4 = \left ( b_1 \omega^1 + b_2 \omega^2 \right ) (E_1) E_4 = b_1 E_4.
		\end{align*}
		Following a similar computation, we obtain $\nabla_{E_2}^\perp E_3 = b_2 E_4, \nabla_{E_1}^\perp E_4 = - b_1 E_3,\nabla_{E_2}^\perp E_4 = - b_2 E_3$ and conclude that \eqref{eq:ConnectionFormNormal} holds.
		
		Next, from the flat normal bundle hypothesis, we know that $R^\perp(E_1, E_2) E_3 = 0$. On the other hand, using \eqref{eq:ConnectionFormSurface} and \eqref{eq:ConnectionFormNormal}, we have
		\begin{align*}
			R^\perp(E_1, E_2) E_3 = & \nabla_{E_1}^\perp \nabla_{E_2}^\perp E_3 - \nabla_{E_2}^\perp \nabla_{E_1}^\perp E_3 - \nabla_{[E_1, E_2]}^\perp E_3 \\
			= & \nabla_{E_1}^\perp(b_2 E_4) - \nabla_{E_2}^\perp(b_1 E_4) - \nabla_{\left(\nabla_{E_1} E_2 - \nabla_{E_2} E_1\right)}^\perp E_3 \\
			= & E_1(b_2) E_4 + b_2 \nabla_{E_1}^\perp E_4 - E_2(b_1) E_4 - b_1 \nabla_{E_2}^\perp E_4 - a_1 \nabla_{E_1}^\perp E_3 - a_2 \nabla_{E_2}^\perp E_3 \\
			= & E_1(b_2) E_4 - b_2 b_1 E_3 - E_2(b_1) E_4 + b_1 b_2 E_3 - a_1 b_1 E_4- a_2 b_2 E_4 \\
			= & \left( E_1(b_2) - E_2(b_1) - a_1 b_1- a_2 b_2 \right)E_4,
		\end{align*}
		and this implies $\eqref{eq:FlatNormalBundleEquation}$.
		
		Item \ref{item:gaussianCurvature}: To prove \eqref{eq:GaussEquation}, we recall that $K = \left \langle R(E_1, E_2) E_2, E_1 \right \rangle$ and taking into account \eqref{eq:ConnectionFormSurface}, we compute
		\begin{align*}
			R(E_1, E_2) E_2 = & \nabla_{E_1} \nabla_{E_2} E_2 - \nabla_{E_2}\nabla_{E_1} E_2 - \nabla_{[E_1, E_2]} E_2\\
			= & \nabla_{E_1}(a_2 E_1) - \nabla_{E_2}(a_1 E_1) - \nabla_{\left(\nabla_{E_1} E_2 - \nabla_{E_2} E_1\right)} E_2\\
			= & E_1(a_2) E_1 + a_2 \nabla_{E_1} E_1 - E_2(a_1) E_1 - a_1 \nabla_{E_2} E_1 - a_1 \nabla_{E_1} E_2 - a_2 \nabla_{E_2} E_2\\
			= & \left ( E_1(a_2) - E_2(a_1) - a_1^2 - a_2^2 \right ) E_1,
		\end{align*}
		which implies \eqref{eq:GaussEquation}.
		
		Item \ref{item:fundamentalEquations}: In order to prove $\eqref{eq:GaussianCurvatureFromEquation}$, we study the Gauss equation \eqref{eq:GaussEquationInSpaceForms}. Using \eqref{eq:SecondFundamentalForm}, we have
		\begin{align*}
			\epsilon = & \langle R(E_1, E_2)E_2, E_1 \rangle + \langle B(E_2, E_1), B(E_1, E_2) \rangle - \langle B(E_1, E_1), B(E_2, E_2) \rangle \\
			 = & \langle R(E_1,E_2)E_2,E_1 \rangle - k_1 k_2 + \alpha^2,
		\end{align*}
		which is equivalent to \eqref{eq:GaussianCurvatureFromEquation}.
		
		Studying the Codazzi equation, we deduce \eqref{eq:E1kappa2}, \eqref{eq:E2kappa1}, \eqref{eq:E1alpha} and \eqref{eq:E2alpha}.
	
		Choosing $X = E_1$ and $Y = Z = E_2$ in \eqref{eq:CodazziEquationInSpaceForms} and using \eqref{eq:SecondFundamentalForm}, \eqref{eq:ConnectionFormSurface}, \eqref{eq:ConnectionFormNormal}, we get
		\begin{align*}
			& \left ( \nabla ^\perp _{E_1} B \right ) (E_2, E_2) = \left ( \nabla ^\perp _{E_2} B \right ) (E_1, E_2) \\
			\Leftrightarrow & \nabla ^\perp _{E_1} B(E_2, E_2) - 2 B \left ( \nabla _{E_1} E_2, E_2 \right ) = \nabla ^\perp _{E_2} B(E_1, E_2) - B \left ( \nabla _{E_2} E_1, E_2 \right ) - B \left ( E_1, \nabla _{E_2} E_2 \right ) \\
			\Leftrightarrow & \nabla ^\perp _{E_1} (k_2 E_3 - \alpha E_4) - 2 B(a_1 E_1, E_2) = - B(-a_2 E_2, E_2) - B(E_1, a_2 E_1) \\
			\Leftrightarrow & E_1 (k_2) E_3 + k_2 \nabla ^\perp _{E_1} E_3 - E_1 (\alpha) E_4 - \alpha \nabla ^\perp _{E_1} E_4 = a_2 (k_2 E_3 - \alpha E_4) - a_2 (k_1 E_3 + \alpha E_4) \\
			\Leftrightarrow & E_1 (k_2) E_3 + b_1 k_2 E_4 - E_1 (\alpha) E_4 + \alpha b_1 E_3 = a_2 (k_2 - k_1) E_3 - 2 \alpha a_2 E_4 \\
			\Leftrightarrow & \bigl ( E_1 (k_2) - a_2 (k_2 - k_1) + \alpha b_1 \bigr ) E_3 - \bigl ( E_1 (\alpha) - 2 \alpha a_2 - b_1 k_2 \bigr ) E_4 = 0
		\end{align*}
		and thus \eqref{eq:E1kappa2} and \eqref{eq:E1alpha} hold.
		
		Choosing $X = Z = E_1$ and $Y = E_2$ in \eqref{eq:CodazziEquationInSpaceForms} and using \eqref{eq:SecondFundamentalForm}, \eqref{eq:ConnectionFormSurface}, \eqref{eq:ConnectionFormNormal}, we obtain
		\begin{align*}
			& \left ( \nabla ^\perp _{E_1} B \right ) (E_2, E_1) = \left ( \nabla ^\perp _{E_2} B \right ) (E_1, E_1) \\
			\Leftrightarrow & \nabla ^\perp _{E_1} B(E_2, E_1) - B \left ( \nabla _{E_1} E_2, E_1 \right ) - B \left ( E_2, \nabla _{E_1} E_1 \right ) = \nabla ^\perp _{E_2} B(E_1, E_1) - 2 B \left ( \nabla _{E_2} E_1, E_1 \right ) \\
			\Leftrightarrow & - B(a_1 E_1, E_1) - B(E_2, - a_1 E_2) = \nabla ^\perp _{E_2} (k_1 E_3 + \alpha E_4) - 2 B(-a_2 E_2, E_1) \\
			\Leftrightarrow & - a_1 (k_1 E_3 + \alpha E_4) + a_1 (k_2 E_3 - \alpha E_4) = E_2 (k_1) E_3 + k_1 \nabla ^\perp _{E_2} E_3 + E_2 (\alpha) E_4 + \alpha \nabla ^\perp _{E_2} E_4 \\
			\Leftrightarrow & a_1 (k_2 - k_1) E_3 - 2 \alpha a_1 E_4 = E_2 (k_1) E_3 + k_1 b_2 E_4 + E_2 (\alpha) E_4 - \alpha b_2 E_3 \\
			\Leftrightarrow & \bigl ( E_2 (k_1) - \alpha b_2 - a_1 (k_2 - k_1) \bigr ) E_3 + \bigl ( E_2 (\alpha) + k_1 b_2 + 2 \alpha a_1 \bigr ) E_4 = 0,
		\end{align*}
		which implies \eqref{eq:E2kappa1} and \eqref{eq:E2alpha}.
		
		Item \ref{item:WSurface}: Finally, since $M$ is a W-surface, there exists $W: \mathbb{R}^2 \to \mathbb{R}$, $W = W \left (x^1, x^2 \right )$, which satisfies $\left( \partial W / \partial x^1 \right)^2 + \left( \partial W / \partial x^2 \right)^2 > 0$ on $M$ and $W(f, K) = 0$. Thus
		\[ \left \lbrace 
		\begin{array}{l}
			E_1(W(f, K)) = 0\\[5pt]
			E_2(W(f, K)) = 0
		\end{array} \right.
		\Leftrightarrow \left \lbrace 
		\begin{array}{l}
			E_1(f) \dfrac{\partial W}{\partial x^1}(f, K) + E_1(K) \dfrac{\partial W}{\partial x^2}(f, K) = 0\\[15pt]
			E_2(f) \dfrac{\partial W}{\partial x^1}(f, K) + E_2(K) \dfrac{\partial W}{\partial x^2}(f, K) = 0
		\end{array} \right. .
		\]
		We obtain a system of linear equations in the variables $\left ( \partial W / \partial x^1 \right ) (f, K)$ and $\left ( \partial W / \partial x^2 \right ) (f, K)$. Since $\left( \partial W / \partial x^1 \right)^2 + \left( \partial W / \partial x^2 \right)^2 \neq 0$ at any point of $M$, the system admits a non-trivial solution and this leads to $E_1 (f) E_2 (K) = E_2 (f) E_1 (K)$, which is equivalent to \eqref{eq:WSurfaceEquaition}.
	\end{proof}
	
	Under the additional hypothesis of biconservativity, the W-surfaces with flat normal bundle have new properties.
	
	\begin{proposition} \label{th:propertiesBiconservativeSurface}
		Let $\varphi: M^2 \to N^4(\epsilon)$ be a biconservative W-surface with flat normal bundle. Assume that $H \neq 0$, $\grad f \neq 0$ and $A_3 \neq f \id$ at any point. Then, the following hold
		\begin{align}
			& (k_1 + f) E_1(f) = -f \alpha b_1, \label{eq:E1MeanCurvature}\\
			&(k_2 + f) E_2(f) = f \alpha b_2, \label{eq:E2MeanCurvature} \\
			&E_1(k_1) = \frac {k_2 - k_1} {2f} \left(E_1(f) - 2 f a_2 \right), \label{eq:E1kappa1} \\
			&E_2(k_2) = -\frac {k_2 - k_1} {2f} \left( E_2(f) + 2 f a_1 \right), \label{eq:E2kappa2} \\
			&E_1(K) = 6f E_1(f) - 4a_2 \left ( f^2 - K + \epsilon \right ), \label{eq:E1GaussianCurvature} \\
			&E_2(K) = 6f E_2(f) + 4a_1 \left ( f^2 - K + \epsilon \right ). \label{eq:E2GaussianCurvature}
		\end{align}
	\end{proposition}

	\begin{proof}
		We begin by expressing \eqref{eq:biharmonicEquationTangentPart} in the frame $\{E_i\}_{i=1}^4$. Using \eqref{eq:MeanCurvatureNormalized}, \eqref{eq:ShapeOperator3} and Proposition \ref{th:propertiesOurSurface} we obtain
		\begin{align*}
			2 \trace A_{\nabla_{(\cdot)}^\perp H} (\cdot) = & 2 \sum_{i=1}^2 A_{\nabla_{E_i}^\perp H} E_i = 2 \sum_{i=1}^2 A_{\nabla_{E_i}^\perp (f E_3)} E_i = 2 \sum_{i=1}^2 E_i(f) A_3 E_i + f A_{\nabla_{E_i}^\perp E_3} E_i\\
			= & 2 \bigl ( E_1(f) A_3 E_1 + E_2(f) A_3 E_2 + f b_1 A_4 E_1 + f b_2 A_4 E_2 \bigr )\\
			= & 2 \bigl ( E_1(f) k_1 E_1 + E_2(f) k_2 E_2 + f b_1 \alpha E_1 - f b_2 \alpha E_2 \bigr )\\
			= & 2 \bigl ( E_1(f) k_1 + f b_1 \alpha \bigr ) E_1 + 2 \bigl ( E_2(f) k_2 - f b_2 \alpha \bigr ) E_2
		\end{align*}
		and
		\[
		\grad |H|^2 = \grad f^2 = 2f \grad f = 2f \bigl ( E_1(f) E_1 + E_2(f) E_2 \bigr ). 
		\]
		Combining these expressions, we deduce that \eqref{eq:biharmonicEquationTangentPart} is equivalent to
		\[
		2 \bigl ( (k_1 + f) E_1(f)  + f \alpha  b_1 \bigr ) E_1 + 2 \bigl ( (k_2 + f) E_2(f) - f \alpha b_2 \bigr ) E_2 = 0.
		\]
		Thus, the biconservativity of $M$ is equivalent to \eqref{eq:E1MeanCurvature} and \eqref{eq:E2MeanCurvature}.
		
		Next, we compute $E_1(k_1)$ and $E_2(k_2)$. Differentiating $2f = k_1 + k_2$ along $E_1$ and substituting \eqref{eq:E1kappa2} and \eqref{eq:E1MeanCurvature}, we have
		\begin{align*} 
			& E_1(k_1) = 2 E_1(f) - E_1(k_2) = 2 E_1(f) - a_2(k_2 - k_1) + \alpha b_1 \\
			\Leftrightarrow & 2f E_1(k_1) = 4f E_1(f) - 2 f a_2(k_2 - k_1) + 2 f \alpha b_1 \\
			\Leftrightarrow & 2f E_1(k_1) = 4f E_1(f) - 2 f a_2(k_2 - k_1) - 2(k_1 + f) E_1(f) \\
			\Leftrightarrow & 2f E_1(k_1) = (4f - 2k_1 - 2f) E_1(f) - 2 f a_2(k_2 - k_1) \\
			\Leftrightarrow & 2f E_1(k_1) = (k_2 - k_1) E_1(f) - 2 f a_2(k_2 - k_1).   
		\end{align*}
		Thus, taking into account that $f \neq 0$ at any point of $M$, we obtain \eqref{eq:E1kappa1}.
		
		Similarly, differentiating $2f = k_1 + k_2$ in the direction of $E_2$ and using \eqref{eq:E2kappa1} and \eqref{eq:E2MeanCurvature}, we obtain
		\begin{align*} 
			& E_2(k_2) = 2 E_2(f) - E_2(k_1) = 2 E_2(f) - a_1(k_2 - k_1) - \alpha b_2 \\
			\Leftrightarrow & 2f E_2(k_2) = 4f E_2(f) - 2 f a_1(k_2 - k_1) - 2 f \alpha b_2 \\
			\Leftrightarrow & 2f E_2(k_2) = 4f E_2(f) - 2 f a_1(k_2 - k_1) - 2 (k_2 + f) E_2(f) \\
			\Leftrightarrow & 2f E_2(k_2) = (4f - 2k_2 - 2f) E_2(f) - 2 f a_1(k_2 - k_1) \\
			\Leftrightarrow & 2f E_2(k_2) = -(k_2 - k_1) E_2(f) - 2 f a_1(k_2 - k_1).   
		\end{align*}
		Thus, since $f \neq 0$ at any point of $M$, we obtain \eqref{eq:E2kappa2}.
		
		Now, using \eqref{eq:GaussianCurvatureFromEquation}, we compute the derivatives of $K$ in the directions $E_1$ and $E_2$. Using \eqref{eq:E1kappa2}, \eqref{eq:E1alpha} and \eqref{eq:E1MeanCurvature}, we have
		\begin{align*}
			E_1(K) = & E_1(\epsilon + k_1 k_2 -\alpha^2) = E_1(k_1) k_2 + k_1 E_1(k_2) - 2 \alpha E_1(\alpha) \\
			= & k_2 E_1(2f - k_2) + k_1  E_1(k_2) - 2 \alpha E_1(\alpha) = 2 k_2 E_1(f) + (k_1 - k_2) E_1(k_2) - 2 \alpha E_1(\alpha) \\
			= & 2 k_2 E_1(f) - a_2(k_2 - k_1)^2 + \alpha b_1 (k_2 - k_1)  - 4 \alpha^2 a_2 - 2 \alpha b_1 k_2 \\
			= & 2 k_2 E_1(f) - a_2\left( (k_2 - k_1)^2 + 4 \alpha^2 \right) - \alpha b_1 (k_1 + k_2) \\
			= & 2 k_2 E_1(f) - a_2\left( (k_2 + k_1)^2 - 4k_1 k_2 + 4 \alpha^2 \right) - 2 \alpha b_1 f \\
			= & 2 k_2 E_1 (f) + 2 (k_1 + f) E_1 (f) - 4 a_2 \left ( f^2 - k_1 k_2 + \alpha^2 \right ).
		\end{align*}
		Using again \eqref{eq:GaussianCurvatureFromEquation}, we obtain \eqref{eq:E1GaussianCurvature}.
		
		Following similar computation and using \eqref{eq:E2kappa1}, \eqref{eq:E2alpha} and \eqref{eq:E2MeanCurvature}, we have
		\begin{align*}
			E_2(K) = & E_2(\epsilon + k_1 k_2 -\alpha^2) = E_2(k_1) k_2 + k_1 E_2(k_2) - 2 \alpha E_2(\alpha) \\
			= & k_2 E_2(k_1) + k_1  E_2(2f - k_1) - 2 \alpha E_2(\alpha) = (k_2 - k_1)E_2(k_1) + 2 k_1 E_2(f)  - 2 \alpha E_2(\alpha) \\
			= & a_1(k_2 - k_1)^2 + \alpha b_2 (k_2 - k_1) + 2 k_1 E_2(f) + 4 	\alpha^2 a_1 + 2 \alpha b_2 k_1 \\
			= & a_1 \left( (k_2 - k_1)^2  + 4 \alpha^2 \right) + \alpha b_2 (k_2 + k_1) + 2 k_1 E_2(f)\\
			= & a_1 \left( (k_2 + k_1)^2 - 4k_1k_2  + 4 \alpha^2 \right) + 2 \alpha b_2 f + 2 k_1 E_2(f)\\
			= & 2 k_1 E_2(f) + 2 (k_2 + f)E_2(f) + 4 a_1 \left( f^2 - k_1k_2  +  \alpha^2 \right),
		\end{align*}
		and, using \eqref{eq:GaussianCurvatureFromEquation}, we conclude that \eqref{eq:E2GaussianCurvature} holds.
	\end{proof}
	
	Recall that the pseudo-umbilical biconservative surfaces in $N^4 (\epsilon)$ are CMC. For the simplicity of the exposition, by non-CMC biconservative surfaces we understand biconservative surfaces such that $H \neq 0$, $\grad f \neq 0$ and $A_3 \neq f \id$ at any point.
	
	Under a small technical assumption, we see that the surfaces we are studying have a key property.
	
	\begin{lemma} \label{th:constraintNormalConnection}
		Let $\varphi: M^2 \to N^4(\epsilon)$ be a non-CMC biconservative W-surface with flat normal bundle. Assume that $3 f^2 + K - \epsilon \neq 0$ at any point. Then
		\begin{equation*}
			\langle \nabla_{E_1}^\perp E_3, E_4 \rangle \langle \nabla_{E_2}^\perp E_3, E_4 \rangle = 0 \text{ on } M.	
		\end{equation*}
		Moreover, on $M$, we have 
		\begin{equation} \label{eq:VanishingBiconservativeWSurfaceFinal}
			\left \langle \grad f, [E_1, E_2] \right \rangle = 0,
		\end{equation} 
	
		\begin{equation} \label{eq:VanishingOfLieBracketsOfFandK}
			[E_1, E_2] (f) = [E_1, E_2] (K) = 0, \quad \text{on } M,
		\end{equation}
	
		\begin{equation} \label{eq:VanishingE1a1PlusE2a2}
			E_1(a_1) + E_2(a_2) = 0
		\end{equation}
		and
		\begin{equation} \label{eq:DifferenceBracketFinal}
			\alpha \left ( b_2E_1(f) - b_1E_2(f) \right ) = 0.
		\end{equation}
	\end{lemma}
	
	\begin{proof}
		In order to prove \eqref{eq:VanishingBiconservativeWSurfaceFinal}, we use the condition \eqref{eq:WSurfaceEquaition}, which is equivalent to $E_1 (f) E_2 (K) = E_2 (f) E_1 (K)$, together with \eqref{eq:E1GaussianCurvature} and \eqref{eq:E2GaussianCurvature}. We obtain
		\[
		E_1(f) \Bigl ( 6f E_2(f) + 4a_1 \left( f^2 - K + \epsilon \right) \Bigr ) = E_2(f) \Bigl ( 6f E_1(f) - 4a_2\left( f^2 - K + \epsilon \right) \Bigr ),  
		\]
		which yields
		\begin{equation} \label{eq:VanishingBiconservativeWSurface}
			\left ( f^2 - K + \epsilon \right ) \bigl ( a_1 E_1(f) + a_2 E_2(f) \bigr ) = 0.
		\end{equation}
		Further, from \eqref{eq:GaussianCurvatureFromEquation} and the fact that $k_1 \neq k_2$ at any point of $M$, we obtain that
		\[
		f^2 - K + \epsilon = \left(\dfrac{k_1 + k_2}{2}\right)^2 - k_1k_2 + \alpha^2 = \left(\dfrac{k_1 - k_2}{2}\right)^2 + \alpha^2 > 0,
		\]
		that is 
		\begin{equation} \label{eq:IneqFandK}
			f^2 - K + \epsilon \neq 0, \quad \text{at any point of } M.
		\end{equation}
		Thus, \eqref{eq:VanishingBiconservativeWSurface} is equivalent to \eqref{eq:VanishingBiconservativeWSurfaceFinal}. 
		
		Now, we prove that $[E_1, E_2] (K)$ and $[E_1, E_2] (f)$ vanish on $M$. Using the fact that $\nabla$ is torsion-free, from \eqref{eq:ConnectionFormSurface}, \eqref{eq:E1GaussianCurvature}, \eqref{eq:E2GaussianCurvature} and \eqref{eq:VanishingBiconservativeWSurfaceFinal}, we have
		\begin{equation*} \label{eq:LieBracketFFromTorsion}
			[E_1, E_2](f) = \left ( \nabla_{E_1} E_2 - \nabla_{E_2} E_1 \right )(f) = a_1E_1(f) + a_2 E_2(f) = 0.
		\end{equation*} 
		Similarly,
		\begin{align*}
			[E_1, E_2](K) = & \left (\nabla_{E_1} E_2 - \nabla_{E_2} E_1 \right )(K) = a_1E_1(K) + a_2 E_2(K) \\ 
			= & 6 a_1 f E_1(f) - 4a_1a_2 \left ( f^2 - K + \epsilon \right ) + 6 a_2 f E_2(f) + 4 a_1 a_2 \left ( f^2 - K + \epsilon \right ) \notag \\
			= & 6f \bigl ( a_1 E_1(f) + a_2 E_2(f) \bigr ). \notag \\
			= & 0.
		\end{align*}
		
		In the following, we compute $[E_1, E_2] (K)$ using the definition of the Lie bracket. Differentiating \eqref{eq:E2GaussianCurvature} along $E_1$, we have
		\begin{align*}
			E_1(E_2(K)) = & E_1 \Bigl ( 6f E_2(f) + 4a_1 \left ( f^2 - K + \epsilon \right ) \Bigr ) \\
			= & 6 E_1(f) E_2(f) + 6f E_1(E_2(f)) + 4 E_1(a_1) \left ( f^2 - K + \epsilon \right ) + 4a_1 \bigl ( 2f E_1(f) - E_1(K) \bigr ),
		\end{align*}
		which, using \eqref{eq:E1GaussianCurvature}, becomes
		\begin{align}
			E_1(E_2(K)) = & 6 E_1(f) E_2(f) + 6f E_1(E_2(f)) + 4 \left ( f^2 - K + \epsilon \right ) E_1(a_1) \label{eq:E1E2GaussianCurvature}\\
			& - 16 f a_1 E_1(f) + 16 a_1 a_2 \left ( f^2 - K + \epsilon \right ). \notag
		\end{align}
		Now, differentiating \eqref{eq:E1GaussianCurvature} along $E_2$, we have
		\begin{align*}
			E_2(E_1(K)) = & E_2 \Bigl ( 6f E_1(f) - 4a_2 \left ( f^2 - K + \epsilon \right ) \Bigr ) \\
			= & 6 E_2(f) E_1(f) + 6f E_2(E_1(f)) - 4 E_2(a_2) \left ( f^2 - K + \epsilon \right ) - 4a_2 \bigl ( 2f E_2(f) - E_2(K) \bigr ),
		\end{align*}
		and, using \eqref{eq:E2GaussianCurvature}, we obtain
		\begin{align}
			E_2(E_1(K)) = & 6 E_1(f) E_2(f) + 6f E_2(E_1(f)) - 4 \left ( f^2 - K + \epsilon \right ) E_2(a_2) \label{eq:E2E1GaussianCurvature}\\
			& + 16 f a_2 E_2(f) + 16 a_1 a_2 \left ( f^2 - K + \epsilon \right ). \notag
		\end{align}
		Combining \eqref{eq:E1E2GaussianCurvature} and \eqref{eq:E2E1GaussianCurvature} and using \eqref{eq:VanishingBiconservativeWSurfaceFinal}, we find that
		\begin{equation} \label{eq:BracketGaussianCurvature}
			[E_1, E_2](K) = 6 f [E_1, E_2](f) + 4 \left ( f^2 - K + \epsilon \right )(E_1(a_1) + E_2(a_2)).
		\end{equation}
		Using \eqref{eq:VanishingOfLieBracketsOfFandK}, \eqref{eq:IneqFandK} and \eqref{eq:BracketGaussianCurvature} we get \eqref{eq:VanishingE1a1PlusE2a2}.
		
		In the following we compute $[E_1, E_2] (k_1)$ and $[E_1, E_2] (k_2)$ in two ways. First we compute them using the fact that $\nabla$ is torsion-free.
		
		From \eqref{eq:ConnectionFormSurface}, \eqref{eq:E2kappa1} and \eqref{eq:E1kappa1}, we have
		\begin{align}
			[E_1, E_2](k_1) = & \left ( \nabla_{E_1} E_2 - \nabla_{E_2}E_1 \right )(k_1) = a_1E_1(k_1) + a_2 E_2(k_1) \label{eq:BracketConnectionkappa1}\\
			= & \dfrac{a_1(k_2 - k_1)}{2f}(E_1(f) - 2f a_2) + a_2\left(a_1(k_2 - k_1) + \alpha b_2\right) \notag \\
			= & \dfrac{a_1(k_2 - k_1)}{2f}E_1(f) + \alpha a_2 b_2. \notag
		\end{align}
		Similarly, using \eqref{eq:ConnectionFormSurface}, \eqref{eq:E1kappa2} and \eqref{eq:E2kappa2}, we get
		\begin{align}
			[E_1, E_2](k_2) = & \left ( \nabla_{E_1} E_2 - \nabla_{E_2}E_1 \right ) (k_2) = a_1E_1(k_2) + a_2 E_2(k_2) \label{eq:BracketConnectionkappa2}\\
			= & a_1\left(a_2(k_2 - k_1) - \alpha b_1\right) -\dfrac{a_2(k_2 - k_1)}{2f}\left(E_2(f) + 2fa_1 \right)\notag\\
			= & - \dfrac{ a_2 (k_2 - k_1)}{2f}E_2(f) - \alpha a_1 b_1 \notag.
		\end{align}
		Now, we compute $[E_1, E_2] (k_1)$ and $[E_1, E_2] (k_2)$ using the definition of the Lie bracket.
		
		Differentiating \eqref{eq:E2kappa1} in the direction of $E_1$ and using \eqref{eq:E1kappa2}, \eqref{eq:E1alpha}, \eqref{eq:E1kappa1}, we obtain 
		\begin{align*}
			E_1(E_2(k_1)) = & E_1 \bigl ( a_1(k_2 - k_1) + \alpha b_2 \bigr )\\
			= & E_1(a_1)(k_2 - k_1) + a_1 \bigl ( E_1(k_2) - E_1(k_1) \bigr ) + E_1(\alpha) b_2 + \alpha E_1(b_2)\\
			= & (k_2 - k_1) E_1(a_1) + a_1a_2(k_2 - k_1) - \alpha a_1 b_1 - \dfrac{a_1(k_2 - k_1) }{2f}E_1(f)\\
			& + a_1 a_2 (k_2 - k_1)  + 2\alpha a_2 b_2 + k_2 b_1 b_2  + \alpha E_1(b_2)\\
			= & (k_2 - k_1) E_1(a_1) + 2a_1a_2(k_2 - k_1) - \dfrac{a_1(k_2 - k_1) }{2f}E_1(f)\\
			& - \alpha a_1 b_1 + 2\alpha a_2 b_2 + k_2 b_1 b_2 + \alpha E_1(b_2).
		\end{align*}
		Differentiating \eqref{eq:E1kappa1} in the direction of $E_2$ and using \eqref{eq:E2kappa1} and \eqref{eq:E2kappa2}, we obtain
		\begin{align*}
			E_2(E_1(k_1)) = & E_2 \left((k_2 - k_1)\left(\dfrac{E_1(f)}{2f} - a_2 \right)\right)\\
			= & \bigl ( E_2(k_2) - E_2(k_1) \bigr ) \left(\dfrac{E_1(f)}{2f} - a_2\right)\\
			& + (k_2 - k_1)\left(\dfrac{2f E_2(E_1(f)) - 2 E_1(f) E_2(f)}{4 f^2} - E_2(a_2)\right)\\
			= & \left(-\dfrac{k_2 - k_1}{2f} \left( E_2(f) + 2 f a_1 \right) - a_1(k_2 - k_1) - \alpha b_2 \right) \left(\dfrac{E_1(f)}{2f} - a_2\right)\\
			& + \dfrac{k_2 - k_1}{2 f}E_2(E_1(f)) - \dfrac{k_2 - k_1}{2 f^2}E_1(f) E_2(f) - (k_2 - k_1) E_2(a_2)\\
			= & -\dfrac{3(k_2 - k_1)}{4f^2} E_1(f)E_2(f) + \dfrac{a_2(k_2 - k_1)}{2f}E_2(f) - \dfrac{a_1(k_2 - k_1)}{f} E_1(f)\\
			& +2a_1 a_2(k_2 - k_1) - \dfrac{\alpha b_2}{2f}E_1(f) + \alpha a_2 b_2 + \dfrac{k_2 - k_1}{2 f}E_2(E_1(f))\\
			&  - (k_2 - k_1) E_2(a_2).
		\end{align*}
		Thus
		\begin{align*}
			[E_1, E_2](k_1) = & (k_2 - k_1) \bigl ( E_1(a_1) + E_2(a_2) \bigr ) - \dfrac{k_2 - k_1}{2f} \left(a_1 E_1(f) + a_2 E_2(f)\right) \\
			& + \alpha (a_2 b_2 - a_1 b_1) + k_2 b_1 b_2 + \alpha E_1(b_2) +\dfrac{3(k_2 - k_1)}{4f^2} E_1(f)E_2(f)\\
			& + \dfrac{a_1(k_2 - k_1)}{f} E_1(f) + \dfrac{\alpha b_2}{2f}E_1(f) - \dfrac{k_2 - k_1}{2 f}E_2(E_1(f))\\ 
		\end{align*}
		From \eqref{eq:VanishingBiconservativeWSurfaceFinal} and \eqref{eq:VanishingE1a1PlusE2a2}, we get
		\begin{align}
			[E_1, E_2](k_1) = &\dfrac{3(k_2 - k_1)}{4f^2} E_1(f)E_2(f) - \dfrac{k_2 - k_1}{2 f}E_2(E_1(f)) + \alpha (a_2 b_2 - a_1 b_1)  
			\label{eq:BracketDefinitionkappa1}\\
			& + k_2 b_1 b_2 + \alpha E_1(b_2) + \dfrac{a_1(k_2 - k_1)}{f} E_1(f) + \dfrac{\alpha b_2}{2f}E_1(f). \notag
		\end{align}
		Next we compute $[E_1, E_2](k_2)$. Differentiating \eqref{eq:E2kappa2} in the direction of $E_1$ and using \eqref{eq:E1kappa2} and \eqref{eq:E1kappa1}, we obtain
		\begin{align*}
			E_1(E_2(k_2)) = & E_1\left(-(k_2 - k_1) \left( \dfrac{E_2(f)}{2f} + a_1 \right)\right)\\
			= & \bigl ( -E_1(k_2) + E_1(k_1) \bigr ) \left(\dfrac{E_2(f)}{2f} + a_1\right)\\
			& - (k_2 - k_1)\left(\dfrac{2f E_1(E_2(f)) - 2 E_2(f) E_1(f)}{4 f^2} + E_1(a_1)\right)\\
			= & \left(-a_2(k_2 - k_1) + \alpha b_1 + \frac {k_2 - k_1} {2f} \left( E_1(f) - 2fa_2 \right) \right) 
			\left(\dfrac{E_2(f)}{2f} + a_1\right)\\
			& - \dfrac{k_2 - k_1}{2 f}E_1(E_2(f)) + \dfrac{k_2 - k_1}{2 f^2}E_1(f) E_2(f) - (k_2 - k_1) E_1(a_1)\\
			= & \dfrac{3(k_2 - k_1)}{4f^2} E_1(f)E_2(f) + \dfrac{a_1(k_2 - k_1)}{2f}E_1(f) - \dfrac{a_2(k_2 - k_1)}{f} E_2(f)\\
			& -2a_1 a_2(k_2 - k_1) + \dfrac{\alpha b_1}{2f}E_2(f) + \alpha a_1 b_1 - \dfrac{k_2 - k_1}{2 f}E_1(E_2(f))\\
			&  - (k_2 - k_1) E_1(a_1).
		\end{align*}
		Differentiating \eqref{eq:E1kappa2} in the direction of $E_2$ and using \eqref{eq:E2kappa1}, \eqref{eq:E2alpha}, \eqref{eq:E2kappa2}, we have 
		\begin{align*}
			E_2(E_1(k_2)) = & E_2 \bigl ( a_2(k_2 - k_1) - \alpha b_1 \bigr )\\
			= & E_2(a_2)(k_2 - k_1) + a_2 \bigl ( E_2(k_2) - E_2(k_1) \bigr ) - E_2(\alpha) b_1 - \alpha E_2(b_1)\\
			= & (k_2 - k_1) E_2(a_2) - \dfrac{a_2(k_2 - k_1)}{2f} E_2(f) - a_1 a_2(k_2 - k_1) - a_1 a_2(k_2 - k_1)\\
			& - \alpha a_2 b_2 +2\alpha a_1 b_1 + k_1 b_1 b_2 - \alpha E_2(b_1) \\
			= & (k_2 - k_1) E_2(a_2) - \dfrac{a_2(k_2 - k_1)}{2f} E_2(f) - 2a_1 a_2(k_2 - k_1) \\
			& - \alpha a_2 b_2 +2\alpha a_1 b_1 + k_1 b_1 b_2 - \alpha E_2(b_1).
		\end{align*}
		Then, 
		\begin{align*}
			[E_1, E_2](k_2) = & \dfrac{3(k_2 - k_1)}{4f^2} E_1(f)E_2(f) + \dfrac{(k_2 - k_1)}{2f} \bigl ( a_1 E_1(f) + a_2 E_2(f) \bigr ) - \dfrac{a_2(k_2 - k_1)}{f}  E_2(f)\\
			& + \dfrac{\alpha b_1}{2f}E_2(f) + \alpha ( a_2 b_2 - a_1 b_1) - \dfrac{k_2 - k_1}{2 f}E_1(E_2(f))\\
			& - (k_2 - k_1) \bigl ( E_1(a_1)+ E_2(a_2) \bigr ) - k_1 b_1 b_2 + \alpha E_2(b_1).
		\end{align*}
		Moreover, using \eqref{eq:VanishingBiconservativeWSurfaceFinal} and \eqref{eq:VanishingE1a1PlusE2a2} we get
		\begin{align}
			[E_1, E_2](k_2) = & \dfrac{3(k_2 - k_1)}{4f^2} E_1(f)E_2(f) - \dfrac{a_2(k_2 - k_1)}{f}E_2(f) + \dfrac{\alpha b_1}{2f}E_2(f) \label{eq:BracketDefinitionkappa2}\\		
			& + \alpha (a_2 b_2 - a_1 b_1) - \dfrac{k_2 - k_1}{2 f}E_1(E_2(f)) - k_1 b_1 b_2 + \alpha E_2(b_1) \notag
		\end{align}
		Combining the two expressions of $[E_1, E_2] (k_1)$ given in \eqref{eq:BracketConnectionkappa1} and \eqref{eq:BracketDefinitionkappa1}, we obtain
		\begin{align} \label{eq:FinalConditionBracketkappa1}
			& (k_2 - k_1) \bigl ( 3E_1(f)E_2(f) - 2fE_2(E_1(f)) + 2 f a_1  E_1(f) \bigr ) \\
			& + 4 f^2 \bigl ( -\alpha a_1 b_1 +  k_2 b_1 b_2 + \alpha E_1(b_2) \bigr ) + 2 f\alpha b_2E_1(f) = 0. \notag
		\end{align}
		Combining the two expressions of $[E_1, E_2] (k_2)$ given in \eqref{eq:BracketConnectionkappa2} and \eqref{eq:BracketDefinitionkappa2}, we get
		\begin{align} \label{eq:FinalConditionBracketkappa2}
			& (k_2 - k_1) \bigl ( 3E_1(f)E_2(f) - 2fE_1(E_2(f)) - 2 f a_2  E_2(f) \bigr ) \\
			& + 4 f^2 \bigl ( \alpha a_2 b_2 -  k_1 b_1 b_2 + \alpha E_2(b_1) \bigr ) +  2f\alpha b_1E_2(f) = 0. \notag
		\end{align}
		Now, subtracting \eqref{eq:FinalConditionBracketkappa2} from \eqref{eq:FinalConditionBracketkappa1}, we have
		\begin{align*}
			& (k_2 - k_1) \bigl ( 2f[E_1, E_2](f) + 2 f (a_1  	E_1(f) + a_2  E_2(f)) \bigr )\\
			& + 4 f^2 \bigl ( -\alpha( a_1 b_1 +  a_2 b_2) + (k_1 + k_2) b_1 b_2 + \alpha \bigl ( E_1(b_2)- E_2(b_1) \bigr ) \bigr )\\
			& + 2 f\alpha \bigl ( b_2E_1(f) -b_1E_2(f) \bigr ) = 0,
		\end{align*}
		Using \eqref{eq:FlatNormalBundleEquation}, \eqref{eq:VanishingBiconservativeWSurfaceFinal} and \eqref{eq:VanishingOfLieBracketsOfFandK}, we obtain
		\begin{equation} \label{eq:DifferenceBracketInitial}
			4 f^2 b_1 b_2 + \alpha \bigl ( b_2E_1(f) - b_1E_2(f) \bigr ) = 0.
		\end{equation}
		Multiplying \eqref{eq:DifferenceBracketInitial} by $(k_1 + f) (k_2 + f)$ and using \eqref{eq:E1MeanCurvature} and \eqref{eq:E2MeanCurvature}, we get
		\begin{align*}
			& 4f^2 b_1 b_2(k_1 + f) (k_2 + f) + \alpha b_2 E_1(f) (k_1 + f) (k_2 + f)\\
			& - \alpha b_1 E_2(f) (k_1 + f) (k_2 + f)  = 0\\
			\Leftrightarrow	& 4 f^2 b_1 b_2 \left ( k_1 k_2 + (k_1 + k_2) f + f^2 \right ) - f \alpha ^2 b_1 b_2 (k_2 + f) - f \alpha^2 b_1 b_2 (k_1 + f) = 0 \\
			\Leftrightarrow & 4 f^2 b_1 b_2 \left ( 3 f^2 + k_1 k_2 \right ) - 4 f^2 b_1 b_2 \alpha^2 = 0 \\
			\Leftrightarrow & 4 f^2 b_1 b_2 \left ( 3 f^2 + k_1 k_2 - \alpha^2 \right ) = 0 \\
			\Leftrightarrow & b_1 b_2 \left ( 3 f^2 + k_1 k_2 - \alpha^2 \right ) = 0.
		\end{align*}
		Substituting \eqref{eq:GaussianCurvatureFromEquation}, we obtain
		\begin{equation*} \label{eq:FinalEquationBiconservativeSurface}
			b_1 b_2 \left ( 3f^2 + K - \epsilon \right ) = 0, \quad \text{on } M.
		\end{equation*}
		Taking into account the fact that $3f^2 + K - \epsilon \neq 0$ at any point of $M$ from the hypothesis, we deduce that $b_1 b_2 = 0$, which, using \eqref{eq:ConnectionFormNormal}, is equivalent to $\langle \nabla_{E_1}^\perp E_3, E_4 \rangle \langle \nabla_{E_2}^\perp E_3, E_4 \rangle = 0$. Replacing this in \eqref{eq:DifferenceBracketInitial}, we obtain \eqref{eq:DifferenceBracketFinal}.
	\end{proof}

	\begin{remark}
		The hypothesis of $3 f^2 + K - \epsilon \neq 0$ at any point was only used to obtain the fact that $b_1 b_2 = 0$ on $M$.
	\end{remark}
	
	The conclusion of Lemma \ref{th:constraintNormalConnection} can be rephrased as $b_1 b_2 = 0$ on $M$. If $b_1 = b_2 = 0$ on $M$, then \eqref{eq:ConnectionFormNormal} implies that $\nabla^\perp E_3 = 0$, that is $M$ is PNMC. The PNMC biconservative surfaces in $4$-dimensional space forms were classified in \cite{NistorOniciucTurgaySen2023} and \cite{NistorRusu2024}. Since we are interested in the non-PNMC case, then, eventually restricting $M$, we further assume that $b_1^2 + b_2^2 > 0$ on $M$. 
	
	\begin{proposition} \label{th:normalConnection}
		Let $\varphi: M^2 \to N^4(\epsilon)$ be a non-CMC biconservative W-surface with flat normal bundle. Assume that $3 f^2 + K - \epsilon \neq 0$ at any point. If $M$ is non-PNMC, then $\left \langle \nabla_{E_1}^\perp E_3, E_4 \right \rangle \neq 0$ at any point of $M$ and $\left \langle \nabla_{E_2}^\perp E_3, E_4 \right \rangle = 0$ on $M$. 
	\end{proposition}
	
	\begin{proof}
		Since $b_1 b_2 = 0$ and $b_1^2 + b_2^2 > 0$ on $M$, we have either
		\[
		b_1 = 0 \quad \text{on } M \quad \text{and} \quad b_2 \neq 0 \quad \text{at any point of } M,
		\]
		or, since $M$ is connected, 
		\[
		b_1 \neq 0 \quad \text{at any point of } M \quad \text{and} \quad b_2 = 0 \quad \text{on } M.
		\] 
		On the other hand, it is easy to check that interchanging $E_1$ and $E_2$ leaves the set of all previously obtained equations unchanged. Therefore, we have only one case and we can choose
		\[
		b_1 \neq 0 \quad \text{at any point of } M \quad \text{and} \quad b_2 = 0 \quad \text{on } M,
		\]
		which represents, using \eqref{eq:ConnectionFormNormal}, the conclusion.
	\end{proof}
	From \eqref{eq:DifferenceBracketFinal} and the hypotheses that $b_1 \neq 0$ and $b_2 = 0$, we obtain
	\begin{equation} \label{eq:alphaE2fZero}
		\alpha E_2(f)=0 \quad \text{on } M.
	\end{equation}
	Suppose by way of contradiction that $\alpha = 0$ on $M$, or on an open subset. It follows from \eqref{eq:E1alpha} that $k_2 b_1 = 0$ on $M$ and, since $b_1 \neq 0$ at any point, we obtain that $k_2 = 0$ on $M$. Hence, from \eqref{eq:E1MeanCurvature} and \eqref{eq:E2MeanCurvature}, $f E_1(f) = f E_2(f) = 0$ on $M$. Since $f$ cannot vanish on $M$, we obtain that $E_1 (f) = E_2 (f) = 0$ on $M$, that is $\grad f = 0$ on $M$, contradiction.
	
	Therefore, eventually restricting $M$ and using \eqref{eq:alphaE2fZero}, we further assume that 
	\begin{equation} \label{eq:alphaNeq0E2fZero}
		\alpha \neq 0 \quad \text{at any point of } M \quad \text{and} \quad E_2 (f) = 0 \quad \text{on } M.
	\end{equation}
	We note that $\alpha \neq 0$ is equivalent to $A_4 \neq 0$.
	
	\begin{remark}
		We note that if $M$ is a non-CMC biconservative surface which is PNMC, then one can have $A_4 = 0$ but, in this case, we have a reduction of the codimension (see \cite{NistorOniciucTurgaySen2023} and \cite{NistorRusu2024}). We note that, in general, the codimension of a non-minimal surface in $N^4 (\epsilon)$ can be reduced if and only if $A_4 = 0$ and it is PNMC. When a non-CMC biconservative surface is non-PNMC, the codimension cannot be reduced. Moreover, we have seen that the case $A_4 = 0$ cannot occur.
	\end{remark}
	
	\begin{remark}
		In the case of PNMC biconservative surfaces in $4$-dimensional space forms it is known that $\grad f$ is an eigenvector of $A_3$, see \cite{FetcuLoubeauOniciuc2021} and \cite{SenTurgay2018}. In our case, when the surface is non-PNMC, as $\grad f \neq 0$ at any point and $E_2 (f) = 0$ on $M$, this fact remains true, that is, up to the sign,
		\[
		E_1 = \frac {\grad f} {|\grad f|}.
		\]
	\end{remark}

	Moreover, from \eqref{eq:VanishingBiconservativeWSurfaceFinal} we get that 
	\begin{equation} \label{eq:a1Zero}
		a_1 = 0 \quad \text{on } M.
	\end{equation}
	Now, assume by way of contradiction that $a_2 = 0$ on $M$, or on an open subset. From \eqref{eq:GaussEquation} we obtain that $K = 0$ on $M$. Then, using \eqref{eq:E1GaussianCurvature} we find that $f E_1 (f) = 0$, which is a contradiction since neither $f$, nor $E_1 (f)$ can vanish on $M$. Therefore, eventually restricting $M$, we further assume that
	\begin{equation*} \label{eq:a2Neq0}
		a_2 \neq 0 \quad \text{at any point of } M,
	\end{equation*}
	that is $\nabla _{E_2} E_2 \neq 0$ at any point.
	
	From \eqref{eq:FlatNormalBundleEquation}, \eqref{eq:E2kappa1}, \eqref{eq:E2alpha}, \eqref{eq:E2kappa2}, \eqref{eq:E2GaussianCurvature} and \eqref{eq:VanishingE1a1PlusE2a2} we obtain
	\begin{equation} \label{eq:VanishingOfTheDerivativesInTheDirOfE2}
		E_2 (b_1) = E_2 (k_1) = E_2 (\alpha) = E_2 (k_2) = E_2 (K) = E_2 (a_2) = 0, \quad \text{on } M.
	\end{equation}
	
	Assume that we are in the hypotheses of Proposition \ref{th:normalConnection}. Let $p_0 \in M$ be an arbitrary fixed point. We consider $\left \{ \psi_s \right \}_{s \in \mathbb R}$ the flow of $E_1$ around $p_0$ and $\gamma = \gamma (t)$ the integral curve of $E_2$ with $\gamma (0) = p_0$. We define the following local chart
	\[
	X^f (s, t) = \psi_s (\gamma (t)) = \psi_{\gamma (t)} (s).
	\]
	We have
	\begin{align*}
		X^f (0, t) =& \gamma (t), \quad \text{for any } t \\
		X^f_t (0, t) =& E_2 (0, t), \quad \text{for any } t \\
		X^f_s (s, t) =& E_1 (s, t), \quad \text{for any } (s, t)
	\end{align*}
	Now, we determine the expression of the metric on $M$ in this local chart.
	\begin{proposition} \label{th:metricOnM}
		Let $\varphi: M^2 \to N^4(\epsilon)$ be a non-CMC biconservative W-surface with flat normal bundle. Assume that $3 f^2 + K - \epsilon \neq 0$, $\nabla^\perp E_3 \neq 0$, $A_4 \neq 0$ and $\nabla_{E_2} E_2 \neq 0$ at any point. Then, around any point, there exist local coordinates $(s, t)$ such that $a_2 = a_2 (s)$ and 
		\begin{equation*}
			g(s, t) = ds^2 + g_{22} (s) dt^2,
		\end{equation*}
		where $g_{22} = g_{22} (s)$ is a positive solution of the following ODE 
		\begin{equation*} \label{eq:differentialEquationMetric}
			\frac {d g_{22}} {ds} = - 2 a_2 g_{22}.
		\end{equation*} 
		Moreover, 
		\[
		E_1 = \frac \partial {\partial s} = \grad s \quad \text{and} \quad E_2 = \frac 1 {\sqrt {g_{22}}} \frac \partial {\partial t}.
		\]
	\end{proposition}

	\begin{proof}
		In the local chart $X^f$, the Riemannian metric of $M^2$ can be written as
		\[
		g = g_{11} ds^2 + 2 g_{12} ds dt + g_{22} dt^2,
		\]
		where $g_{11} = g_{11} (s, t)$, $g_{12} = g_{12} (s, t)$ and $g_{22} = g_{22} (s, t)$ are smooth functions. We have
		\begin{align*}
			g_{11} (s, t) =& \left | X^f_s (s, t) \right |^2 = |E_1 (s, t)|^2 = 1,\\
			g_{12} (0, t) =& \left \langle X^f_s (0, t), X^f_t (0, t) \right \rangle = \langle E_1 (0, t), E_2 (0, t) \rangle = 0,\\
			g_{22} (0, t) =& \left | X^f_t (0, t) \right |^2 = | E_2 (0, t) |^2 = 1,
		\end{align*}
		for any $s$ and $t$.
		
		Suppose that $E_2 = a X^f_s + b X^f_t$. We have
		\[
		\left \langle E_2, X^f_s \right \rangle = \langle E_2, E_1 \rangle = 0.
		\]
		On the other hand,
		\[
		\left \langle E_2, X^f_s \right \rangle = \left \langle a X^f_s + b X^f_t, X^f_s \right \rangle = a g_{11} + bg_{12} = a + b g_{12}.
		\]	
		Thus,
		\[
		a = -b g_{12},
		\]
		and
		\[
		E_2 = b \left ( X_t^f - g_{12}X_s^f \right ).
		\]
		We know that
		\[
		1 = |E_2|^2 = b^2\left ( g_{22} - 2g_{12}^2 + g_{12}^2 g_{11} \right ) = b^2 \left ( g_{22} - g_{12}^2 \right )
		\]
		and since $g_{22} - g_{12}^2 = g_{22} g_{11} - g_{12}^2 > 0$, without loss of generality, we can assume that 
		\[
		b = \frac 1 {\sqrt{g_{22} - g_{12}^2}}
		\]
		and obtain
		\begin{equation} \label{eq:E1E2InLocalChart}
			E_1 = X^f_s \quad \text{and} \quad E_2 = \frac 1 {\sqrt {g_{22} - g_{12}^2}} \left ( X^f_t - g_{12} X^f_s \right ).
		\end{equation}
		Let $f(s, t) = \left ( f \circ X^f \right )(s, t)$ be the mean curvature function expressed in this local chart. Since $E_2 (f) = 0$, from \eqref{eq:E1E2InLocalChart} we obtain
		\begin{equation}\label{eq:relationNaturalBasis}
			X_t^f(f) = g_{12} X_s^f(f).
		\end{equation}
		Combining \eqref{eq:VanishingOfLieBracketsOfFandK}, \eqref{eq:alphaNeq0E2fZero}, \eqref{eq:E1E2InLocalChart} and \eqref{eq:relationNaturalBasis}, we obtain
		\begin{equation*}
			0 = [E_1, E_2](f) = E_2 (E_1 (f)),
		\end{equation*}
		that is
		\begin{align*}
			0 = & \left ( X^f_t - g_{12} X^f_s \right ) \left ( X^f_s (f) \right ) \\
			  =& X^f_t \left ( X^f_s (f) \right ) - g_{12} X^f_s \left ( X^f_s (f) \right ) \\
			  =& X^f_t \left ( X^f_s (f) \right ) - X^f_s \left ( g_{12} X^f_s (f) \right ) + X^f_s (g_{12}) X^f_s (f) \\
			  =& X^f_t \left ( X^f_s (f) \right ) - X^f_s \left ( X^f_t (f) \right ) + X^f_s (g_{12}) X^f_s (f) \\
			  =& \left [ X^f_t, X^f_s \right ] (f) + X^f_s (g_{12}) E_1 (f).
		\end{align*}
		Using the fact that $\left [ X^f_t, X^f_s \right ] = 0$ and $|E_1 (f)| = |\grad f| \neq 0$, we get $X^f_s (g_{12}) = 0$ everywhere, which implies that
		\begin{equation*}
			g_{12} (s, t) = g_{12} (0, t) = 0,
		\end{equation*}
		for any $s$ and $t$. Therefore,
		\[
		g (s, t) = ds^2 + g_{22} (s, t) dt^2,
		\]
		\[
		E_1 = X^f_s = \frac \partial {\partial s} \quad \text{and} \quad E_2 = \frac 1 {\sqrt {g_{22}}} X^f_t = \frac 1 {\sqrt {g_{22}}} \frac \partial {\partial t}.
		\]
		Next, we find a differential equation which defines $g_{22}$.
		
		From \eqref{eq:ConnectionFormSurface}, we have
		\begin{align*}
			0 = & \nabla_{E_1} E_2 = \nabla_{ \frac{\partial}{\partial s}} \left( \frac{1}{\sqrt{g_{22}}} \frac{\partial}{\partial t}\right) = \frac{\partial}{\partial s}\left(\frac{1}{\sqrt{g_{22}}}\right)\frac{\partial}{\partial t} + \frac{1}{\sqrt{g_{22}}} \nabla_{ \frac{\partial}{\partial s}} \frac{\partial}{\partial t} \\
			= & -\frac{1}{2} \frac{\partial g_{22}}{\partial s} \frac{1}{\sqrt{g_{22}^3}} \frac{\partial}{\partial t} + \frac{1}{\sqrt{g_{22}}} \left(\Gamma_{12}^1\frac{\partial}{\partial s} + \Gamma_{12}^2\frac{\partial}{\partial t}\right),
		\end{align*}
		which implies that
		\begin{equation} \label{eq:CristofellSymbolsMetric1}
			\Gamma_{12}^1 = 0 \quad \text{and} \quad \Gamma_{12}^2 = \frac{1} {2g_{22}} \frac{\partial g_{22}}{\partial s}.
		\end{equation} 
		We also know from \eqref{eq:ConnectionFormSurface} that
		\[
		\nabla_{E_2} E_1 = - a_2 E_2 = - \frac{a_2}{\sqrt{g_{22}}} \frac{\partial}{\partial t}.	
		\]
		On the other hand,
		\[
		- \frac{a_2}{\sqrt{g_{22}}} \frac{\partial}{\partial t} = \nabla_{E_2} E_1 = \nabla_{\frac{1}{\sqrt{g_{22}}} \frac{\partial}{\partial t}} \dfrac{\partial}{\partial s} = \frac{1}{\sqrt{g_{22}}} \left(\Gamma_{12}^1 \dfrac{\partial}{\partial s} + \Gamma_{12}^2 \dfrac{\partial}{\partial t} \right) = \frac{1}{\sqrt{g_{22}}} \Gamma_{12}^2 \dfrac{\partial}{\partial t}.
		\]
		Thus,
		\begin{equation} \label{eq:CristofellSymbolsMetric2}
			\Gamma_{12}^2 = - a_2.
		\end{equation}
		Combining \eqref{eq:CristofellSymbolsMetric1} and \eqref{eq:CristofellSymbolsMetric2}, we obtain that
		\begin{equation}\label{eq:MetricDifferentialEquation}
			\frac {\partial g_{22}} {\partial s} = - 2 a_2 g_{22}.
		\end{equation}
		Computing the other Christofell symbols, we get no additional information. Also, note that from \eqref{eq:VanishingOfTheDerivativesInTheDirOfE2}, we obtain that the function $a_2$ depends only on the parameter $s$.
		
		In the following we want to find the positive solutions of \eqref{eq:MetricDifferentialEquation}. We have
		\[
		\frac {\partial} {\partial s} \bigl ( \ln (g_{22}(s, t)) \bigr ) = - 2 a_2 (s).
		\]
		We consider an arbitrarily fixed primitive $A_2$ of $a_2$. Thus, we get that 
		\[
		\ln (g_{22} (s, t)) = - 2 A_2 (s) + 2c_1 (t),
		\]
		where $c_1$ is a smooth function. Therefore,
		\[
		g_{22} (s, t) = e^{2c_1 (t)} e^{-2 A_2 (s)}
		\]
		and the metric $g$ becomes
		\[
		g (s, t) = ds^2 + e^{2c_1 (t)} e^{-2 A_2 (s)} dt^2.
		\]
		If we consider the change of coordinates
		\[
		(s, t) \to \left ( \tilde s = s, \tilde t = \int _0 ^t e^{c_1 (\tau)} d\tau \right ),
		\]
		the metric $g$ takes the form
		\[
		g = d \tilde s^2 + e^{-2 A_2 \left ( \tilde s \right )} d \tilde t^2.
		\]
		In conclusion, we obtain $\tilde g_{22} = \tilde g_{22} \left ( \tilde s \right ) = e^{-2 A_2 \left ( \tilde s \right )}$. In fact, $\tilde g_{22}$ is uniquely determined up to a multiplicative positive constant, but this constant does not play an essential role since we can always make a simple transformation and include it in the new parameter $\tilde t$.
		
		Moreover,
		\[
		E_1 = \frac \partial {\partial \tilde s} \quad \text{and} \quad E_2 = \frac 1 {\sqrt {\tilde g_{22}}} \frac \partial {\partial \tilde t}.
		\]
		For a simpler notation we redenote $\left ( \tilde s, \tilde t \right ) \to (s, t)$.
	\end{proof}

	Summarizing all information we have until now, we obtain that a non-CMC (non-PNMC) biconservative W-surface with flat normal bundle must satisfy the following first order ODE system.
	
	\begin{theorem} \label{th:directTheorem}
		Let $\varphi: M^2 \to N^4(\epsilon)$ be a non-CMC biconservative W-surface with flat normal bundle. Assume that $3 f^2 + K - \epsilon \neq 0$, $\nabla^\perp E_3 \neq 0$, $A_4 \neq 0$ and $\nabla_{E_2} E_2 \neq 0$ at any point. Then, around any point, there exist local coordinates $(s, t)$ such that $f = f(s)$, $k_1 = k_1 (s)$, $k_2 = k_2 (s)$, $\alpha = \alpha (s)$, $a_2 = a_2 (s)$, $b_1 = b_1 (s)$ and $K = K (s)$. Moreover, the tuple $\left ( a_2, f, \alpha, k_2 \right )$ is a solution of the following first order ODE system
		\begin{equation} \label{eq:differentialSystemTangentPart}
			\left \{
			\begin{array}{l}
				\dot a_2 = \epsilon + k_2 (2 f - k_2) - \alpha^2 + a_2^2 \\[5pt]
				\displaystyle \dot f = - \frac {f \alpha b_1} {3 f - k_2} \\[10pt]
				\dot \alpha = 2 \alpha a_2 + k_2 b_1 \\[5pt]
				\dot k_2 = 2 a_2 (k_2 - f) - \alpha b_1
			\end{array} 
			\right .,
		\end{equation}
		where $\dot a_2$, $\dot f$, $\dot \alpha$ and $\dot k_2$ represent the derivatives with respect to $s$ of $a_2$, $f$, $\alpha$ and $k_2$, respectively.
	\end{theorem}

	\begin{proof}
		From \eqref{eq:alphaNeq0E2fZero} and \eqref{eq:VanishingOfTheDerivativesInTheDirOfE2} we obtain that the functions $f$, $k_1$, $k_2$, $\alpha$, $a_2$, $b_1$ and $K$ depend only on the parameter $s$.
		
		Replacing \eqref{eq:GaussianCurvatureFromEquation} and \eqref{eq:a1Zero} in \eqref{eq:GaussEquation} we obtain the first equation of system \eqref{eq:differentialSystemTangentPart}. From \eqref{eq:E1MeanCurvature} we obtain the second equation of the system. We note that $3f - k_2 \neq 0$ since $f$, $\alpha$ and $b_1$ are different from $0$, that is the right hand-side of \eqref{eq:E1MeanCurvature} is different from $0$. The third and last equations of the system are \eqref{eq:E1alpha} and \eqref{eq:E1kappa2}, respectively.
		
		One can check that replacing \eqref{eq:alphaNeq0E2fZero}, \eqref{eq:a1Zero}, \eqref{eq:VanishingOfTheDerivativesInTheDirOfE2} and the result of Proposition \ref{th:normalConnection} in the relations of Propositions \ref{th:propertiesOurSurface}, \ref{th:propertiesBiconservativeSurface} and Lemma \ref{th:constraintNormalConnection} not yet used gives no additional information.
	\end{proof}
	
	\begin{remark}
		As we will see in the following, we do not need the surface to be Weingarten, but the slightly weaker condition that $\grad f$ is parallel to $\grad K$.
	\end{remark}

	In the following we provide a converse of Theorem \ref{th:directTheorem}. For this, we first denote $\uu = a_2$, $\vv = b_1$, $\ww = f$, $\xx = \alpha$ and $\yy = k_2$ and we rewrite \eqref{eq:differentialSystemTangentPart}. Let 
	\[
	\Omega = \left \{ (\uu, \ww, \xx, \yy) \in \mathbb R^* \times (0, \infty) \times \mathbb R^* \times \mathbb R \ \middle | \ 3 \ww - \yy \neq 0 \ \text{and} \ 3 \ww^2 + \yy (2 \ww - \yy) - \xx^2 \neq 0 \right \},
	\]
	and $I$ be a real open interval. We define $F_\vv : I \times \Omega \to \mathbb R^4$ by
	\begin{equation*}
		F_\vv (s, \uu, \ww, \xx, \yy) = 
		\begin{pmatrix}
			\epsilon + \yy ( 2 \ww - \yy) - \xx^2 + \uu^2 \\[5pt]
			\displaystyle - \frac {\ww \xx \vv (s)} {3\ww - \yy} \\[10pt]
			2 \xx \uu + \yy \vv (s) \\[5pt]
			2 \uu (\yy - \ww) - \xx \vv (s)
		\end{pmatrix}
	\end{equation*}
	and it is clear that \eqref{eq:differentialSystemTangentPart} is equivalent to the following first order ODE system
	\begin{equation} \label{eq:differentialSystemTangentPartConverse}
		\dot X (s) = F_\vv (s, X(s)), \quad \text{for any } s,
	\end{equation}
	where $\epsilon \in \mathbb R$, $\vv : I \to \mathbb R^*$ is a smooth arbitrarily fixed function and $X (s) = \left ( \uu (s),  \ww (s),  \xx (s),  \yy (s) \right )$.
	
	Since $F_\vv$ is smooth, given an arbitrary initial condition $(s_0, \uu_0, \ww_0, \xx_0, \yy_0) \in I \times \Omega$, the system of equations \eqref{eq:differentialSystemTangentPartConverse} has a unique solution around $s_0$, for any smooth function $\vv$.
	
	If $(\uu (s), \ww (s), \xx (s), \yy (s))$ is a solution of \eqref{eq:differentialSystemTangentPartConverse}, then 
	\[
	\left ( \tilde \uu \left ( \tilde s \right ) = - \uu \left ( - \tilde s \right ), \tilde \ww \left ( \tilde s \right ) = \ww \left ( - \tilde s \right ), \tilde \xx \left ( \tilde s \right ) = \xx \left ( - \tilde s \right ), \tilde \yy \left ( \tilde s \right ) = \yy \left ( - \tilde s \right ) \right )
	\]
	is a solution of \eqref{eq:differentialSystemTangentPartConverse} associated to $F_{\tilde \vv}$, where $\tilde \vv \left ( \tilde s \right ) = - \vv \left ( - \tilde s \right )$. Further, 
	\[
	(\uu (s), \ww (s), - \xx (s), \yy (s))
	\]
	is a solution of \eqref{eq:differentialSystemTangentPartConverse} associated to $F_{- \vv}$.
	
	The above two properties have natural geometric correspondence and from now on we assume that 
	\[
	\dot \ww > 0 \quad \text{and} \quad \xx > 0.
	\]
	We note that $E_1$ and $E_4$ may change their signs and, in this case,
	\[
	E_1 = \frac {\grad f} {|\grad f|}.
	\]
	Consequently, the domain $\Omega$ becomes
	\[
	\Omega = \left \{ (\uu, \ww, \xx, \yy) \in \mathbb R^* \times (0, \infty) \times (0, \infty) \times \mathbb R \ \middle | \ 3 \ww - \yy \neq 0 \ \text{and} \ 3 \ww^2 + \yy (2 \ww - \yy) - \xx^2 \neq 0 \right \}.
	\]
	
	Starting with a solution of \eqref{eq:differentialSystemTangentPartConverse}, in the next result we provide a way to construct non-CMC biconservative W-surfaces with flat normal bundle and satisfying the additional requirements. Thus, we can say that \eqref{eq:differentialSystemTangentPartConverse} represents (all) the compatibility conditions for this class of biconservative surfaces.
	
	\begin{theorem} \label{th:existenceTheorem}
		Let $\vv : I \to \mathbb R^*$ be a smooth function and consider $(\uu, \ww, \xx, \yy)$ a solution of \eqref{eq:differentialSystemTangentPartConverse} defined on $I$. On $I \times \mathbb R$ we define the metric $g (s, t) = ds^2 + g_{22} (s) dt^2$, for any $(s, t) \in I \times \mathbb R$, where $g_{22}$ is a positive solution of 
		\[
		\frac {d g_{22}} {ds} = - 2 \uu g_{22}.
		\]
		Then, there exists a biconservative immersion $\varphi : I \times \mathbb R \to N^4 (\epsilon)$ such that 
		\begin{enumerate} [label = \alph*)]
			\item $3 f^2 + K - \epsilon \neq 0$, $H \neq 0$, $\grad f \neq 0$, $A_3 \neq f \id$, $A_4 \neq 0$ and $\nabla _{E_2} E_2 \neq 0$ at any point of $I \times \mathbb R$; \label{item:existenceA}
			\item $\varphi$ has flat normal bundle and $\grad f$ is parallel to $\grad K$;
			\label{item:existenceB}
			\item $\nabla ^\perp _{E_1} E_3 = \vv E_4 \neq 0$ at any point and $\nabla ^\perp _{E_2} E_3 = 0$ on $I \times \mathbb R$, thus $\varphi$ is non-PNMC. \label{item:existenceC}
		\end{enumerate}
		Moreover, we have $f = f (s) = \ww (s)$, $k_2 = k_2 (s) = \yy (s)$, $\alpha = \alpha (s) = \xx (s)$, $a_2 = a_2 (s) = \uu (s)$ and $b_1 = b_1 (s) = \vv (s)$.
	\end{theorem}
	
	\begin{proof}
		For simplicity, we denote $M = I \times \mathbb R$. Taking into account Proposition \ref{th:metricOnM}, we define the orthonormal frame field tangent to $M$ by 
		\[
		E_1 = \frac{\partial}{\partial s} \quad \text{and} \quad E_2 = \frac{1}{\sqrt{g_{22}}}\frac{\partial}{\partial t}.
		\]
		Further, let $\Upsilon = M^2 \times \mathbb{R}^2 $ be the trivial vector bundle of rank $2$ over $M$. We define $\sigma_3$ and $\sigma_4$ by
		\[
		\sigma_3(p) = (p, (1, 0)) \quad \text{and} \quad \sigma_4(p) = (p, (0, 1)), \quad \text{for any } p \in M,
		\]
		the metric $h$ on $\Upsilon$ by
		\begin{equation*}\label{eq:MetricUpsilon}
			h(\sigma_\alpha, \sigma_\beta) = \langle \sigma_\alpha, \sigma_\beta \rangle = \delta_{\alpha \beta}, \quad \text{for any } \alpha, \beta \in \{ 3, 4 \}, 
		\end{equation*}
		and the connection $\nabla^\Upsilon$ on $\Upsilon$ by
		\begin{equation}\label{eq:ConnectionUpsilon}
			\left \{
			\begin{array}{l}
				\nabla_{E_1}^\Upsilon \sigma_3 = \vv \sigma_4 \\[5pt]
				\nabla_{E_1}^\Upsilon \sigma_4 = -\vv \sigma_3 \\[5pt]
				\nabla_{E_2}^\Upsilon \sigma_3 = \nabla_{E_2}^\Upsilon \sigma_4 = 0
			\end{array}
			\right..
		\end{equation}
		The sections $\sigma_3$ and $\sigma_4$ form the canonical global frame field of $\Upsilon$.
		
		It is easy to check that the pair $\left ( \nabla^\Upsilon, h \right )$ is a Riemannian structure, that is
		\begin{equation*} \label{eq:RiemannianStructure}
			X\left(\langle \sigma, \rho \rangle\right) = \left \langle \nabla_X^\Upsilon \sigma, \rho \right \rangle + \left \langle  \sigma, \nabla_X^\Upsilon \rho \right \rangle, \quad \text{for any } \sigma, \rho \in C(\Upsilon). 
		\end{equation*}
		Now, we compute the curvature tensor field $R^\Upsilon$ on $\Upsilon$. From the definition of $R^\Upsilon$, we have
		\begin{align*}
			R^\Upsilon (E_1, E_2) \sigma_3 = & \nabla_{E_1}^\Upsilon \nabla_{E_2}^\Upsilon \sigma_3 - \nabla_{E_2}^\Upsilon \nabla_{E_1}^\Upsilon \sigma_3 - \nabla_{[E_1, E_2]}^\Upsilon \sigma_3\\
			= & - \nabla_{E_2}^\Upsilon \left( \vv \sigma_4 \right) -  \nabla_{\left [ \frac{\partial}{\partial s}, \frac 1 {\sqrt{g_{22}}} \frac {\partial} {\partial t} \right ]}^\Upsilon \sigma_3\\
			= & - E_2 (\vv) \sigma_4 - \vv \nabla_{E_2}^\Upsilon \sigma_4 - \nabla_{\left ( \frac {\partial} {\partial s} \left ( \frac 1 {\sqrt{g_{22}}} \right ) \frac {\partial} {\partial t} + \frac 1 {\sqrt{g_{22}}} \left [\frac {\partial} {\partial s}, \frac {\partial} {\partial t} \right ] \right )}^\Upsilon \sigma_3\\
			= & - \frac 1 {\sqrt{g_{22}}} \frac{\partial \vv} {\partial t} \sigma_4 + \frac {\dot g_{22}} {2 g_{22} \sqrt {g_{22}}}  \nabla_{\frac {\partial} {\partial t}}^\Upsilon \sigma_3\\
			= & \frac {\dot g_{22}} {2 g_{22}} \nabla _{E_2} ^\Upsilon \sigma_3 \\ 
			= & 0,   
		\end{align*}
		and
		\begin{align*}
			R^\Upsilon (E_1, E_2) \sigma_4 = & \nabla_{E_1}^\Upsilon \nabla_{E_2}^\Upsilon \sigma_4 - \nabla_{E_2}^\Upsilon \nabla_{E_1}^\Upsilon \sigma_4 - \nabla_{[E_1, E_2]}^\Upsilon \sigma_4\\
			= & \nabla_{E_2}^\Upsilon \left( \vv \sigma_3 \right) -  \nabla_{\left [ \frac{\partial}{\partial s}, \frac 1 {\sqrt{g_{22}}} \frac {\partial} {\partial t} \right ]}^\Upsilon \sigma_4\\
			= & E_2 (\vv) \sigma_3 + \vv \nabla_{E_2}^\Upsilon \sigma_3  -  \nabla_{\left ( \frac {\partial} {\partial s} \left ( \frac 1 {\sqrt{g_{22}}} \right ) \frac {\partial} {\partial t} + \frac 1 {\sqrt{g_{22}}} \left [\frac {\partial} {\partial s}, \frac {\partial} {\partial t} \right ] \right )}^\Upsilon \sigma_4\\
			= & \frac 1 {\sqrt{g_{22}}} \frac{\partial \vv} {\partial t} \sigma_3 + \frac {\dot g_{22}} {2 g_{22} \sqrt {g_{22}}} \nabla_{\frac {\partial} {\partial t}}^\Upsilon \sigma_4\\
			= & \frac {\dot g_{22}} {2 g_{22}} \nabla _{E_2} ^\Upsilon \sigma_4 \\
			= & 0.   
		\end{align*}
		Therefore, 
		\[
		R^\Upsilon = 0 \quad \text{on } M.
		\]
		Now, we define $B^\Upsilon : C(TM) \times C(TM) \to C(\Upsilon)$ a symmetric $C^\infty (M)$-bilinear map by
		\begin{equation}\label{eq:SecondFondamentalFormUpsilon}
			\left\lbrace 
			\begin{array}{l}
				B^\Upsilon (E_1, E_1) = (2 \ww - \yy) \sigma_3 + \xx \sigma_4\\ [5pt]
				B^\Upsilon (E_1, E_2) = 0 \\ [5pt]
				B^\Upsilon (E_2, E_2) = \yy \sigma_3 - \xx \sigma_4
			\end{array}
			\right. .
		\end{equation}
		Consider $A_\alpha^\Upsilon \in C(\End(TM))$ given by
		\begin{equation*}\label{eq:ShapeOperatorUpsilonDefinition}
			\langle A_\alpha^\Upsilon E_i, E_j \rangle  = \langle B^\Upsilon (E_i, E_j), \sigma_\alpha \rangle,
		\end{equation*}
		for any $i, j \in \{1, 2\}$ and any $\alpha \in \{3, 4\}$.
		
		The Christofell symbols of the metric $g$ are given by
		\begin{equation*}\label{eq:SymbolsUpsilon}
			\left\lbrace 
			\begin{array}{l}
				\Gamma_{11}^1 = \Gamma_{11}^2 = \Gamma_{12}^1 = \Gamma_{22}^2 = 0\\[5pt]
				\Gamma_{22}^1 = \uu g_{22} \\[5pt]
				\Gamma_{12}^2 = -\uu
			\end{array}
			\right. .
		\end{equation*}
		Now, we can compute the Levi-Civita connection of $M$
		\begin{align*}
			\nabla_{E_1} E_1 = & \nabla_{ \frac{\partial}{\partial s}}\frac{\partial}{\partial s} = \Gamma_{11}^1\frac{\partial}{\partial s} + \Gamma_{11}^2 \frac{\partial}{\partial t} = 0,\\
			\nabla_{E_1} E_2 = & \nabla_{ \frac{\partial}{\partial s}} \left( \frac{1}{\sqrt{g_{22}}} \frac{\partial}{\partial t}\right)\\
							 = & \frac{\partial}{\partial s}\left(\frac{1}{\sqrt{g_{22}}}\right)\frac{\partial}{\partial t} + \frac{1}{\sqrt{g_{22}}} \nabla_{ \frac{\partial}{\partial s}} \frac{\partial}{\partial t}\\
							 = & -\frac{\dot g_{22}} {2 g_{22} \sqrt {g_{22}}} \frac{\partial}{\partial t} + \frac{1}{\sqrt{g_{22}}} \left(\Gamma_{12}^1\frac{\partial}{\partial s} + \Gamma_{12}^2\frac{\partial}{\partial t}\right)\\
							 = & \uu E_2 - \uu E_2 = 0,\\
			\nabla_{E_2} E_1 = & \frac{1}{\sqrt{g_{22}}} \nabla_{ \frac{\partial}{\partial t}} \frac{\partial}{\partial s} = \frac{1}{\sqrt{g_{22}}} \left(\Gamma_{12}^1 \frac{\partial}{\partial s} + \Gamma_{12}^2 \frac{\partial}{\partial t} \right) \\
							 = & - \frac \uu {\sqrt{g_{22}}} \frac{\partial}{\partial t} = - \uu E_2, \\
			\nabla_{E_2} E_2 = & \frac 1 {g_{22}} \nabla _{\frac \partial {\partial t}} \frac \partial {\partial t} \\
							 = & \frac{1}{g_{22}} \left( \Gamma_{22}^1 \frac{\partial}{\partial s} + \Gamma_{22}^2 \frac{\partial}{\partial t}\right) = \uu \frac{\partial}{\partial s} = \uu E_1.
		\end{align*}
		By direct computations, the Gaussian curvature of $M$ is given by
		\begin{equation*}
			K = \left \langle R (E_1, E_2) E_2, E_1 \right \rangle = \dot \uu - \uu^2.
		\end{equation*}	
		Using the first equations of \eqref{eq:differentialSystemTangentPartConverse}, we obtain that 
		\begin{equation*}
			K = \epsilon + \yy (2 \ww - \yy) - \xx^2.
		\end{equation*}
		Now, we check if the fundamental equations are satisfied. For the Gauss equation \eqref{eq:GaussEquationInSpaceForms} we have 
		\begin{align*}
			\epsilon = & \left \langle R (E_1, E_2) E_2, E_1 \right \rangle - \left \langle B^\Upsilon (E_1, E_1), B^\Upsilon (E_2, E_2) \right \rangle + \left \langle B^\Upsilon (E_1, E_2), B^\Upsilon (E_1, E_2) \right \rangle \\
			 		 = & K - \left \langle (2 \ww - \yy) \sigma_3 + \xx \sigma_4, \yy \sigma_3 - \xx \sigma_4 \right \rangle \\
					 = & K - \yy (2 \ww - \yy) + \xx^2,
		\end{align*}
		which is trivially satisfied.
		
		Next, we study the Codazzi equation \eqref{eq:CodazziEquationInSpaceForms}. Choosing $X = Z = E_1$ and $Y =  E_2$ and taking into account \eqref{eq:ConnectionUpsilon} and \eqref{eq:SecondFondamentalFormUpsilon}, we obtain
		\begin{align*}
			\left ( \nabla ^\Upsilon _{E_1} B^\Upsilon \right ) (E_2, E_1) =& \nabla ^\Upsilon _{E_1} B^\Upsilon (E_2, E_1) - B^\Upsilon \left ( \nabla _{E_1} E_2, E_1 \right ) - B^\Upsilon \left ( E_2, \nabla _{E_1} E_1 \right ) \\
			= & 0
		\end{align*}
		and
		\begin{align*}
			\left ( \nabla ^\Upsilon _{E_2} B^\Upsilon \right ) (E_1, E_1) =& \nabla ^\Upsilon _{E_2} B^\Upsilon (E_1, E_1) - 2 B ^\Upsilon \left ( \nabla _{E_2} E_1, E_1 \right ) \\
			=& \nabla ^\Upsilon _{E_2} \bigl ( (2 \ww - \yy) \sigma_3 + \xx \sigma_4 \bigr ) - 2 B ^\Upsilon (- \uu E_2, E_1) \\
			=& E_2 (2 \ww - \yy) \sigma_3 + (2 \ww - \yy) \nabla ^\Upsilon _{E_2} \sigma_3 + E_2 (\xx) \sigma_4 + \xx \nabla ^\Upsilon _{E_2} \sigma_4 \\
			=& 0.
		\end{align*}
		Thus, $\left ( \nabla ^\Upsilon _{E_1} B^\Upsilon \right ) (E_2, E_1) = \left ( \nabla ^\Upsilon _{E_2} B ^\Upsilon \right ) (E_1, E_1)$.
		
		Choosing $X = E_1$ and $Y = Z = E_2$ and taking into account \eqref{eq:ConnectionUpsilon}, \eqref{eq:SecondFondamentalFormUpsilon}, the third and the fourth equations of \eqref{eq:differentialSystemTangentPartConverse}, we have
		\begin{align*}
			\left ( \nabla ^\Upsilon _{E_1} B ^\Upsilon \right ) (E_2, E_2) =& \nabla ^\Upsilon _{E_1} B ^\Upsilon (E_2, E_2) - 2 B ^\Upsilon \left ( \nabla _{E_1} E_2, E_2 \right ) \\
			=& \nabla ^\Upsilon _{E_1} (\yy \sigma_3 - \xx \sigma_4) \\
			=& E_1 (\yy) \sigma_3 + \yy \nabla ^\Upsilon _{E_1} \sigma_3 - E_1 (\xx) \sigma_4 - \xx \nabla ^\Upsilon _{E_1} \sigma_4 \\
			=& \dot \yy \sigma_3 + \yy \vv \sigma_4 - \dot \xx \sigma_4 + \xx \vv \sigma_3 \\
			=& \left ( \dot \yy + \xx \vv \right ) \sigma_3 - \left ( \dot \xx - \yy \vv \right ) \sigma_4 \\
			=& 2 \uu (\yy - \ww) \sigma_3 - 2 \uu \xx \sigma_4
		\end{align*}
		and 
		\begin{align*}
			\left ( \nabla ^\Upsilon _{E_2} B ^\Upsilon \right ) (E_1, E_2) =& \nabla ^\Upsilon _{E_2} B ^\Upsilon (E_1, E_2) - B ^\Upsilon \left ( \nabla _{E_2} E_1, E_2 \right ) - B ^\Upsilon \left ( E_1, \nabla _{E_2} E_2 \right ) \\
			=& \uu (\yy \sigma_3 - \xx \sigma_4) - \uu \bigl ( (2 \ww - \yy) \sigma_3 + \xx \sigma_4 \bigr ) \\
			=& 2 \uu (\yy - \ww) \sigma_3 - 2 \uu \xx \sigma_4.
		\end{align*}
		Thus, $\left ( \nabla ^\Upsilon _{E_1} B ^\Upsilon \right ) (E_2, E_2) = \left ( \nabla ^\Upsilon _{E_2} B ^\Upsilon \right ) (E_1, E_2)$. Therefore, the Codazzi equation is satisfied.
		
		It remains to check if the Ricci equation \eqref{eq:RicciEquationInSpaceForms} is satisfied. Since $R^\Upsilon = 0$, the Ricci equation is equivalent to
		\begin{equation*}
			A^\Upsilon_3 \circ A^\Upsilon_4 = A^\Upsilon_4 \circ A^\Upsilon_3.
		\end{equation*}
		Using the definition of $A^\Upsilon_3$ and $A^\Upsilon_4$ one can easily check that this relation holds.
		
		Since the Gauss, Codazzi and Ricci equations are formally satisfied and $M$ is simply connected, from the Fundamental Theorem of Submanifolds (for example, see \cite{DajczerTojeiroBook2019}), we conclude that there exists a unique globally defined isometric immersion $\varphi : M^2 \to N^4(\epsilon)$ and a vector bundle isometry $\phi: \Upsilon \to N_{\varphi} M$ such that
		\[
		\nabla^\perp \phi = \phi \nabla^\Upsilon \quad \text{and} \quad B = \phi \circ B^\Upsilon.
		\]
		Now we have to check if $\varphi$ has the properties \ref{item:existenceA}, \ref{item:existenceB} and \ref{item:existenceC}. 
		
		First, note that 
		\[
		H^\Upsilon = \frac 1 2 \trace B^\Upsilon = \ww \sigma_3 \neq 0, \quad \text{at any point}.
		\]
		From the above formula we obtain that $f = \ww$ and, as $\dot \ww > 0$, we deduce that
		\[
		\grad f = \grad \ww = E_1 (\ww) E_1 + E_2 (\ww) E_2 = \dot \ww \frac \partial {\partial s} \neq 0, \quad \text{at any point}.
		\]
		Moreover, $E_1 = \grad f / |\grad f|$.
		
		Relation $3 f^2 + K - \epsilon \neq 0$, which is the first condition we must check in \ref{item:existenceA}, is equivalent to
		\begin{equation*}
			3 \ww^2 + \yy (2 \ww - \yy) - \xx^2 \neq 0.
		\end{equation*} 
		But this is implicitly ensured by the definition of the domain $\Omega$ of the function $F_\vv$.
			
		Now, we check if $A^\Upsilon _3 \neq \ww \id$. Suppose by way of contradiction that $A^\Upsilon _3 = \ww \id$ on an open subset $U$ of $M$. Then, we have $\ww = \yy$ on $U$. Using the second and the last equations of \eqref{eq:differentialSystemTangentPartConverse}, we obtain that $\xx \vv = 0$ on $U$, which is a contradiction since neither $\xx$, nor $\vv$ can vanish on $M$. Thus, $A^\Upsilon _3 \neq \ww \id$ at any point of an open and dense subset of $M$. Eventually restricting $I$, we obtain that $A^\Upsilon _3 \neq \ww \id$ at any point of $M = I \times \mathbb R$.
		
		Now, we check if $\varphi$ is biconservative. From \eqref{eq:biharmonicEquationTangentPart} we have
		\begin{align*}
			    & 2 \left ( A^\Upsilon _{\nabla^\Upsilon _{E_1} H^\Upsilon} E_1 + A^\Upsilon _{\nabla^\Upsilon _{E_2} H^\Upsilon} E_2 \right ) + \grad \ww^2 \\
			  = & 2 \left ( A^\Upsilon _{E_1 (\ww) \sigma_3 + \ww \nabla^\Upsilon _{E_1} \sigma_3} E_1 + A^\Upsilon _{E_2 (\ww) \sigma_3 + \ww \nabla^\Upsilon _{E_2} \sigma_3} E_2 \right ) \\
			    & + E_1 \left ( \ww^2 \right ) E_1 + E_2 \left ( \ww^2 \right ) E_2 \\
			  = & 2 \left ( \dot \ww A^\Upsilon_3 E_1 + \ww \vv A^\Upsilon_4 E_1 \right ) + 2 \ww \dot \ww E_1 \\
			  = & \left ( 2 \dot \ww (2 \ww - \yy) + 2 \xx \vv \ww + 2 \ww \dot \ww \right ) E_1 \\
			  = & 0
		\end{align*}
		in virtue of the second equation of \eqref{eq:differentialSystemTangentPartConverse}.
		
		Using \eqref{eq:SecondFondamentalFormUpsilon}, we obtain that $A^\Upsilon _4 E_1 = \xx E_1$ and $A^\Upsilon _4 E_2 = - \xx E_2$. Since the function $\xx$ is positive, we deduce that $A^\Upsilon _4 \neq 0$ at any point of $M$.
		
		Since the function $\uu$ cannot vanish, we have $\nabla _{E_2} E_2 = \uu E_1 \neq 0$ at any point of $M$. 
		
		It is then straightforward to check that $\varphi$ has the properties \ref{item:existenceB} and \ref{item:existenceC} of the Theorem \ref{th:existenceTheorem}.
	\end{proof}
	
	\begin{remark}
		Fixing $\vv$ corresponds to prescribing the connection in the normal bundle. Thus, Theorem \ref{th:existenceTheorem} can be seen as an existence result for non-CMC biconservative W-surfaces with flat normal bundle when we prescribe the normal connection.
	\end{remark}

	Theorem \ref{th:existenceTheorem} assures that any solution of \eqref{eq:differentialSystemTangentPartConverse} provides a non-CMC biconservative W-surface with flat normal bundle in $N^4 (\epsilon)$. Consequently, constructing examples of such biconservative surfaces is equivalent to finding solutions of \eqref{eq:differentialSystemTangentPartConverse}. In the following, we present a particular solution of \eqref{eq:differentialSystemTangentPartConverse} which has $\yy = 0$. This solution is a reminiscence of what happens in the case of biconservative hypersurfaces since $\grad f$ is now an eigenvector of $A_3$ corresponding to the eigenvalue $2f$.
	
	\begin{theorem} \label{th:exampleYZero}
		Let $\varphi : M^2 \to N^4 (\epsilon)$ be a non-CMC biconservative W-surface with flat normal bundle. Assume that $3 f^2 + K - \epsilon \neq 0$, $\nabla^\perp E_3 \neq 0$, $A_4 \neq 0$ and $\nabla_{E_2} E_2 \neq 0$ at any point. Then, $M$ satisfies $A_3 (\grad f) = 2 f \grad f$ if and only if, locally, 
		\begin{equation} \label{eq:vSolutionYEqualsZeroInF}
			\vv = - \frac {3 c^3 \dot f} {f^3}
		\end{equation}
		and
		\begin{equation} \label{eq:solutionYEqualsZeroInF}
			\left \{
			\begin{array}{l}
				\displaystyle \uu = \frac 3 2 \frac {\dot f} f \\[10pt]
				\ww = f \\[5pt]
				\xx = c^{-3} f^3 \\[5pt]
				\yy = 0
			\end{array}
			\right .,
		\end{equation}
		where $c$ is a positive constant and $f$ is a positive solution of 
		\begin{equation} \label{eq:secondOrderDifferentialEquationMeanCurvatureCaseYZero}
			\ddot f f - \frac 5 2 \dot f^2 - \frac 2 3 \epsilon f^2 + \frac 2 3 c^{-6} f^8 = 0
		\end{equation}
		such that $\dot f > 0$.
	\end{theorem}
	
	\begin{proof}
		We want to find a solution of \eqref{eq:differentialSystemTangentPartConverse} which satisfies $\yy = 0$, that is we want a solution of 
		\begin{equation} \label{eq:differentialSystemYZero}
			\left \{
			\begin{array}{l}
				\dot \uu = \epsilon - \xx^2 + \uu^2 \\[5pt]
				\displaystyle \dot \ww = - \frac {\xx \vv} 3 \\[10pt]
				\dot \xx = 2 \uu \xx \\[5pt]
				0 = - 2 \uu \ww - \xx \vv 
			\end{array}
			\right ..
		\end{equation}
		We write $\uu$ as $\uu = \dot {\bar \U}$. From the third equation of \eqref{eq:differentialSystemYZero} we obtain 
		\begin{equation*}
			\frac d {ds} (\ln \xx) = 2 \uu,
		\end{equation*}
		which implies that $\xx = e^{c_1} e^{2 \bar \U}$, where $c_1 \in \mathbb R$. Redenoting $c_1= e^{c_1} > 0$, we find that
		\begin{equation*}
			\xx = c_1 e^{2 \bar \U}.
		\end{equation*}
		Replacing $\uu = \dot {\bar \U}$ and $\xx$ in the first equation of \eqref{eq:differentialSystemYZero}, we deduce that $\bar \U$ satisfies 
		\begin{equation*}
			\ddot {\bar \U} = \epsilon - c_1^2 e^{4 \bar \U} + \dot {\bar \U}^2.
		\end{equation*}
		Using the second and the last equations of \eqref{eq:differentialSystemYZero}, we obtain that
		\begin{equation*}
			\dot \ww = \frac 2 3 \uu \ww,
		\end{equation*}
		which implies
		\begin{equation*}
			\ww = c_2 e^{\frac 2 3 \bar \U},
		\end{equation*} 
		where $c_2$ is a positive real constant.
		
		From the fourth equation of \eqref{eq:differentialSystemYZero} we get 
		\begin{equation*}
			\vv = - \frac {2 c_2} {c_1} \dot {\bar \U} e^{- \frac 4 3 \bar \U}.
		\end{equation*}
		Denoting $\U = \bar \U + (\ln c_1) / 2$ and putting $c = c_2 / \sqrt[3] {c_1}$, we obtain
		\begin{equation} \label{eq:vSolutionYEqualsZeroInQ}
			\vv = - 2 c \dot \U e^{- \frac 4 3 \U}
		\end{equation}
		and
		\begin{equation} \label{eq:solutionYEqualsZeroInQ}
			\left \{
			\begin{array}{l}
				\uu = \dot \U \\[5pt]
				\ww = c e^{\frac 2 3 \U} \\[5pt]
				\xx = e^{2 \U} \\[5pt]
				\yy = 0 
			\end{array}
			\right .,
		\end{equation}
		where $c$ is a positive real constant and $\U$ is a solution of 
		\begin{equation} \label{eq:solutionYEqualsZeroSecondOrderODEInQ}
			\ddot \U = \epsilon - e^{4 \U} + \dot \U^2
		\end{equation}
		such that $\dot \U > 0$ at any point.
		
		In the case of PNMC biconservative surfaces it is natural to express the functions $\uu$, $\vv$, $\ww$, $\xx$ and $\yy$ in terms of the mean curvature function $f = f(s)$, see \cite{NistorOniciucTurgaySen2023}, \cite{NistorRusu2024} and \cite{SenTurgay2018}. 
		
		For this, we denote $f = c e^{\frac 2 3 \U}$. Differentiating this relation, we get
		\begin{equation*}
			\dot \U = \frac 3 2 \frac {\dot f} f.
		\end{equation*}
		Then, the conclusion follows.
	\end{proof}
	
	\begin{remark}
		Using standard methods, we can solve the second order ODE \eqref{eq:solutionYEqualsZeroSecondOrderODEInQ} and find $\U = \U(s)$. Then, replacing $\U$ in \eqref{eq:vSolutionYEqualsZeroInQ} and \eqref{eq:solutionYEqualsZeroInQ}, we find explicit formulas for the functions $\uu$, $\vv$, $\ww$, $\xx$ and $\yy$ described in Theorem \ref{th:exampleYZero}. For example, if $\epsilon = 0$, then $\U(s) = \ln \sqrt {S (s)}$,
		\begin{equation*}
			\vv(s) = \pm 2 c (\pm \mathcal A s + \mathcal B) S^{\frac 1 3} (s)
		\end{equation*}
		and
		\begin{equation*}
			\left \{
			\begin{array}{l}
				\uu (s) = \mp (\pm \mathcal A s + \mathcal B) S (s) \\[5pt]
				\ww (s) = c S^{\frac 1 3} (s) \\[5pt]
				\xx (s) = S (s) \\[5pt]
				\yy (s) = 0
			\end{array}
			\right ., \quad \text{for any } s,
		\end{equation*}
		where $\mathcal A > 0$, $\mathcal B \in \mathbb R$ and 
		\begin{equation*}
			S (s) = \frac {\mathcal A} {1 + (\pm \mathcal A s + \mathcal B)^2}.
		\end{equation*}
		If $\epsilon = 1$, then $\U (s) = - \ln \sqrt {S (s)}$, 
		\begin{equation*}
			\vv(s) = \pm 2 c \sqrt {\mathcal A^2 - 1} \cos (\pm 2 s + \mathcal B) S^{- \frac 1 3} (s)
		\end{equation*}
		and
		\begin{equation*}
			\left \{
			\begin{array}{l}
				\uu (s) = \mp \sqrt {\mathcal A^2 - 1} \cos (\pm 2 s + \mathcal B) S^{-1} (s) \\[5pt]
				\ww (s) = c S^{- \frac 1 3} (s) \\[5pt]
				\xx (s) = S^{-1} (s) \\[5pt]
				\yy (s) = 0
			\end{array}
			\right ., \quad \text{for any } s,
		\end{equation*}
		where $\mathcal A > 0$, $\mathcal B \in \mathbb R$ and 
		\begin{equation*}
			S(s) = \mathcal A + \sqrt {\mathcal A^2 - 1} \sin (\pm 2s + \mathcal B).
		\end{equation*}
	\end{remark}
		
	We note that a first integral of \eqref{eq:secondOrderDifferentialEquationMeanCurvatureCaseYZero} is given by
	\begin{equation} \label{eq:firstOrderDifferentialEquationMeanCurvatureCaseYZero}
		\dot f^2 = \frac 4 9 f^2 \left ( C f^3 - c^{-6} f^6 - \epsilon \right ),
	\end{equation}
	where $C \in \mathbb R$, if $\epsilon < 0$ and $C > 0$, if $\epsilon \geq 0$.
	
	We know that the metric on $M$ is given by 
	\begin{equation*}
		g (s, t) = ds^2 + g_{22} (s) dt^2, \quad \text{for any } (s, t) \in I \times \mathbb R,
	\end{equation*}
	where $g_{22}$ is a positive solution of $\dot g_{22} (s) = - 2 \uu (s) g_{22} (s)$.
	
	Taking into account the expression of $\uu$ and the fact that $g_{22}$ is uniquely determined up to multiplicative positive constants, we obtain that 
	\begin{equation*}
		g (s, t) = ds^2 + f^{-3} (s) dt^2.
	\end{equation*}

	It can be also useful to perform a change of coordinates and have $f$ as a parameter (see \cite{NistorOniciucTurgaySen2023} and \cite{SenTurgay2018}). In this case, it is easy to see that the non-CMC biconservative W-surfaces with flat normal bundle which satisfy $A_3 (\grad f) = 2 f \grad f$ given in Theorem \ref{th:exampleYZero} form a $2$-parameter family. For this, we perform the change of coordinates $(s, t) \to (f = f (s), t)$ and obtain that $d f = \dot f (s) ds$. Using \eqref{eq:firstOrderDifferentialEquationMeanCurvatureCaseYZero}, we deduce that 
	\begin{equation*}
		d f^2 = \frac 4 9 f^2 \left ( C f^3 - c^{-6} f^6 - \epsilon \right ) ds^2
	\end{equation*}
	and thus
	\begin{equation} \label{eq:metricSolutionYEqualsZeroInFT}
		g (f, t) = \frac 9 {4 f^2 \left ( C f^3 - c^{-6} f^6 - \epsilon \right )} d f^2 + f^{-3} dt^2, \quad \text{for any } (f, t).
	\end{equation}
	
	In this coordinates system, the functions $\uu$, $\vv$, $\ww$, $\xx$ and $\yy$ from \eqref{eq:vSolutionYEqualsZeroInF} and \eqref{eq:solutionYEqualsZeroInF} are given by 
	\begin{equation} \label{eq:vInTheCaseYEqualsZeroInTheChartFT}
		\vv = - \frac {2 c^3} {f^2} \sqrt {C f^3 - c^{-6} f^6 - \epsilon}
	\end{equation}
	and
	\begin{equation} \label{eq:solutionInTheCaseYEqualsZeroInTheChartFT}
		\left \{
		\begin{array}{l}
			\uu = \sqrt {C f^3 - c^{-6} f^6 - \epsilon} \\[5pt]
			\ww = f \\[5pt]
			\xx = c^{-3} f^3 \\[5pt]
			\yy = 0
		\end{array}
		\right .,
	\end{equation}
	where $c > 0$ and $f$ is positive.
	
	We note that, from the definition of $\Omega$, see \eqref{eq:differentialSystemTangentPartConverse}, the solution must satisfy $3 \ww - \yy \neq 0$ and $3 \ww^2 + \yy (2 \ww - \yy) - \xx^2 \neq 0$ at any point. In our case, except for at most one point, solution \eqref{eq:solutionInTheCaseYEqualsZeroInTheChartFT} satisfies these inequalities. Thus, eventually restricting the domain interval $I$, the two inequalities are satisfied at any point.

	Even if in \eqref{eq:metricSolutionYEqualsZeroInFT} we have two constants, the family of all metrics given by \eqref{eq:metricSolutionYEqualsZeroInFT} form a $1$-parametric family of non-isometric metrics. In order to see this, we perform a further change of coordinates $(f, t) \to \left ( \F = c^{-1} f, \mathcal T = c^{- 3 / 2} t \right )$ and the metric \eqref{eq:metricSolutionYEqualsZeroInFT} can be written as 
	\begin{equation} \label{eq:metricSolutionYEqualsZeroInFTOneConstant}
		g(\F, \mathcal T) = \frac 9 {4 \F^2 \left ( \mathcal C \F^3 - \F^6 - \epsilon \right )} d \F^2 + \F^{-3} d \mathcal T^2,
	\end{equation}
	where $\mathcal C = C c^3$. Moreover, if we fix the domain metric $g$, that is we fix the constant $\mathcal C$, we have a $1$-parameter family of non-CMC and non-PNMC biconservative W-immersions with flat normal bundle indexed by $c$.
	
	\begin{remark}
		The expression of the Gaussian curvature of non-CMC biconservative W-surfaces with flat normal bundle and with $A_3 (\grad f) = 2 f \grad f$ does not depend on that constant $\mathcal C$ from the expression of \eqref{eq:metricSolutionYEqualsZeroInFTOneConstant}, since $K = \epsilon - \F^6$. Thus, after restricting the domain, for two distinct values of the constant $\mathcal C$ we obtain two non-isometric abstract surfaces with the same (non-constant) Gaussian curvature.
	\end{remark}
	
	To determine other biconservative surfaces with specific properties, for example with constant Gaussian curvature, it is useful to rewrite system \eqref{eq:differentialSystemTangentPartConverse} in an equivalent form.
	
	System \eqref{eq:differentialSystemTangentPartConverse} does not explicitly involve the Gaussian curvature $K$, but all the information provided by $K$ are enclosed in this system. As we will see, by including $K$ in \eqref{eq:differentialSystemTangentPartConverse}, specifically by including \eqref{eq:E1GaussianCurvature} and \eqref{eq:GaussianCurvatureFromEquation} as a constraint, we obtain a new system equivalent to \eqref{eq:differentialSystemTangentPartConverse}. It turns out that this new system is more appropriate to fulfill our objective.
	
	So, let
	\[
	\overline \Omega = \left \{ (\uu, \ww, \xx, \yy, \kk) \in \mathbb R^* \times (0, \infty) \times (0, \infty) \times \mathbb R \times \mathbb R \ \middle | \ 3 \ww - \yy \neq 0 \ \text{and} \ 3 \ww^2 + \kk - \epsilon \neq 0 \right \},
	\]
	and $I$ be an open interval. We define $\overline F_\vv : I \times \overline \Omega \to \mathbb R^5$ by
	\begin{equation*}
		\overline F_\vv (s, \uu, \ww, \xx, \yy, \kk) = 
		\begin{pmatrix}
			\kk + \uu^2 \\[5pt]
			\displaystyle - \frac {\ww \xx \vv (s)} {3\ww - \yy} \\[10pt]
			2 \xx \uu + \yy \vv (s) \\[5pt]
			2 \uu (\yy - \ww) - \xx \vv (s) \\[5pt]
			\displaystyle - \frac {6 \ww^2 \xx \vv (s)} {3 \ww - \yy} - 4 \uu \left ( \ww^2 - \kk + \epsilon \right )
		\end{pmatrix}
	\end{equation*}
	and consider the ODE system
	\begin{equation*} 
		\dot X (s) = \overline F_\vv \left ( s, X(s) \right ), \quad \text{for any } s,
	\end{equation*}
	where $\epsilon \in \mathbb R$, $\vv : I \to \mathbb R^*$ is a smooth arbitrarily fixed function and $X (s) = \left ( \uu (s), \ww (s), \xx (s), \yy (s), \kk (s) \right )$.
	
	Now, it is easy to see that
	
	\begin{proposition}
		The differential system \eqref{eq:differentialSystemTangentPartConverse} is equivalent to the following differential system
		\begin{equation} \label{eq:differentialSystemTangentPartWithK}
			\left \{
			\begin{array}{l}
				\dot X (s) = \overline F_\vv \left ( s, X(s) \right ), \\[5pt]
				\kk (s) = \epsilon + \yy (s) \bigl ( 2 \ww (s) - \yy (s) \bigr ) - \xx^2 (s)
			\end{array}
			\right . , \quad \text{for any } s.
		\end{equation}
	\end{proposition}
	
	As we have announced, we present a particular solution of \eqref{eq:differentialSystemTangentPartWithK} and, consequently a solution of \eqref{eq:differentialSystemTangentPartConverse}, with a nice geometric meaning.
	
	\begin{theorem} \label{th:exampleKEqualsEpsilon}
		Let $\varphi : M^2 \to N^4 (\epsilon)$ be a non-CMC biconservative W-surface with flat normal bundle. Assume that $3 f^2 + K - \epsilon \neq 0$, $\nabla^\perp E_3 \neq 0$, $A_4 \neq 0$ and $\nabla_{E_2} E_2 \neq 0$ at any point. Then, $M$ has constant Gaussian curvature $K = \epsilon$ if and only if, locally, 
		\begin{equation} \label{eq:vSolutionKEqualsEpsInTermsOfF}
			\vv = - \frac {\dot f (3 - c f)} {\sqrt c f^{\frac 3 2} \sqrt {2 - c f}}
		\end{equation}
		and
		\begin{equation} \label{eq:solutionKEqualsEpsInTermsOfF}
			\left \{
			\begin{array}{l}
				\displaystyle \uu = \frac 3 2 \frac {\dot f} f \\[10pt]
				\displaystyle \ww = f \\[5pt]
				\displaystyle \xx = \sqrt c f^{\frac 3 2} \sqrt {2 - c f} \\[10pt]
				\displaystyle \yy = c f^2
			\end{array}
			\right .,
		\end{equation}
		where $c$ is a positive real constant and $f$ is a positive solution of 
		\begin{equation} \label{eq:secondOrderODEForF}
			\ddot f f - \frac 5 2 \dot f^2 - \frac 2 3 \epsilon f^2 = 0
		\end{equation}
		which satisfies $\dot f > 0$.
	\end{theorem}

	\begin{proof}
		Suppose that $\kk = \epsilon$. Around any point we consider local coordinates $(s, t)$ given by Proposition \ref{th:metricOnM}. The first equation of system \eqref{eq:differentialSystemTangentPartWithK} becomes 
		\[
		\dot \uu = \epsilon + \uu^2.
		\]
		In the following we solve this equation and find the positive function $\uu$.
		
		\begin{enumerate} [label = \roman*)]
			\item If $\epsilon = 0$, then $\dot \uu = \uu^2$. Since $\uu \neq 0$ at any point, we have
			\begin{equation*}
				\frac {\dot \uu (s)} {\uu^2 (s)} = 1 \Leftrightarrow \frac 1 {\uu (s)} = - s + \mathcal C,
			\end{equation*}
			where $\mathcal C \in \mathbb R$, $s \in I$.
			Thus,
			\begin{equation} \label{eq:uSolutionKEquals0Explicit}
				\uu (s) = \frac 1 {-s + \mathcal C}, \quad \text{for any } s \in I,
			\end{equation}
			where $I = (- \infty, \mathcal C)$.
			
			\item If $\epsilon > 0$, then 
			\begin{equation*}
				\frac {\dot \uu (s)} {\epsilon + \uu^2 (s)} = 1 \Leftrightarrow \frac 1 {\sqrt \epsilon} \arctan \frac {\uu (s)} {\sqrt \epsilon} = s + \mathcal C,
			\end{equation*}
			where $\mathcal C \in \mathbb R$ and $s \in \left ( - \pi / \left ( 2 \sqrt \epsilon \right ) - \mathcal C , \pi / \left ( 2 \sqrt \epsilon \right ) - \mathcal C  \right )$.
			
			Therefore, 
			\begin{equation} \label{eq:uSolutionKEquals1Explicit}
				\uu (s) = \sqrt \epsilon \tan \left ( \sqrt \epsilon (s + \mathcal C) \right ), \quad \text{for any } s \in I, 
			\end{equation}
			where $I = \left ( - \mathcal C, \pi / \left ( 2 \sqrt \epsilon \right ) - \mathcal C \right )$.
			
			\item If $\epsilon < 0$, we distinguish two cases. If $\uu$ is constant, then $\uu (s) = \pm \sqrt {- \epsilon}$, for any $s \in I$.
			
			If $\uu$ is not constant, then, eventually restricting $I$, we have $\epsilon + \uu^2 \neq 0$ at any point. Then, we have
			\begin{equation*}
				\frac {\dot \uu (s)} {\epsilon + \uu^2 (s)} = 1 \Leftrightarrow \frac 1 {2 \sqrt {- \epsilon}} \ln \left | \frac {\uu (s) - \sqrt {- \epsilon}} {\uu (s) + \sqrt {- \epsilon}} \right | = s + \mathcal C \Leftrightarrow \left | \frac {\uu (s) - \sqrt {- \epsilon}} {\uu (s) + \sqrt {- \epsilon}} \right | = e^{2 \sqrt {- \epsilon} (s + \mathcal C)},
			\end{equation*}
			where $\mathcal C \in \mathbb R$ and $s \in I$.
			
			Since the left hand-side of the previous relation does not vanish, we obtain either 
			\begin{equation*}
				\frac {\uu (s) - \sqrt {- \epsilon}} {\uu (s) + \sqrt {- \epsilon}} = e^{2 \sqrt {- \epsilon} (s + \mathcal C)},
			\end{equation*}
			or
			\begin{equation*}
				\frac {\uu (s) - \sqrt {- \epsilon}} {\uu (s) + \sqrt {- \epsilon}} = - e^{2 \sqrt {- \epsilon} (s + \mathcal C)}.
			\end{equation*}
			Therefore, either 
			\begin{equation} \label{eq:uSolutionKEquals-1Explicit1}
				\uu (s) = \frac {\sqrt{- \epsilon} \left ( 1 + e^{2 \sqrt {- \epsilon} (s + \mathcal C)} \right )} {1 - e^{2 \sqrt {- \epsilon} (s + \mathcal C)}}, \quad \text{for any } s \in I,
			\end{equation}
			or
			\begin{equation} \label{eq:uSolutionKEquals-1Explicit2}
				\uu (s) = \frac {\sqrt{- \epsilon} \left ( 1 - e^{2 \sqrt {- \epsilon} (s + \mathcal C)} \right )} {1 + e^{2 \sqrt {- \epsilon} (s + \mathcal C)}}, \quad \text{for any } s \in I,
			\end{equation}
			where $\mathcal C \in \mathbb R$ and, in both cases, $I = (- \infty, - \mathcal C)$.
		\end{enumerate}
		Summarizing, for any value of the sectional curvature $\epsilon \in \mathbb R$ of the target, the differential equation $\dot \uu = \epsilon + \uu^2$ has explicit solutions. Let $\uu$ be a solution of this equation.
		
		Using the second and fifth equations of \eqref{eq:differentialSystemTangentPartWithK} and the fact that $\ww$ must be positive, we obtain 
		\begin{equation*}
			6 \dot \ww \ww - 4 \uu \ww^2 = 0 \Leftrightarrow 3 \frac {d \ww^2} {ds} = 4 \uu \ww^2 \Leftrightarrow \frac d {ds} \left ( \ln \ww^2 \right ) = \frac 4 3 \uu.
		\end{equation*}
		Therefore, considering an arbitrarily fixed primitive $\U$ of $\uu$, we get
		\begin{equation*}
			\ww = e^{c_1} e^{\frac 2 3 \U},
		\end{equation*}
		where $c_1 \in \mathbb R$. Since $\uu > 0$, we have
		\[
		\dot \ww = \frac 2 3 e^{c_1} \uu e^{\frac 2 3 \U} > 0.
		\]
		From the second equation of \eqref{eq:differentialSystemTangentPartWithK} we obtain that 
		\begin{equation} \label{eq:expressionXVs}
			\xx \vv = - 2 e^{c_1} \uu e^{\frac 2 3 \U} + \frac 2 3 \uu \yy.
		\end{equation}
		Substituting \eqref{eq:expressionXVs} in the forth equation of \eqref{eq:differentialSystemTangentPartWithK} we obtain
		\begin{align*}
			 \dot \yy =& 2 \uu \yy - 2 e^{c_1} \uu e^{\frac 2 3 \U} + 2 e^{c_1} \uu e^{\frac 2 3 \U} - \frac 2 3 \uu \yy \\
			 		  =& \frac 4 3 \uu \yy.
		\end{align*}
		Multiplying this relation by $e^{ - 4 \U / 3}$ we obtain 
		\begin{equation*}
			\frac d {ds} \left ( \yy e^{- \frac 4 3 \U} \right ) = 0,
		\end{equation*}
		which means that $\yy e^{- 4 \U / 3}$ is a first integral, that is
		\begin{equation*}
			\yy = c_2 e^{\frac 4 3 \U},
		\end{equation*}
		where $c_2 \in \mathbb R$.
		
		Now we multiply the third equation of \eqref{eq:differentialSystemTangentPartWithK} by $\xx$ and using \eqref{eq:expressionXVs}, we obtain
		\begin{align*}
			& \dot \xx \xx = 2 \uu \xx^2 + \yy \xx \vv \\
			\Leftrightarrow & \frac d {ds} \left ( \xx^2 \right ) - 4 \uu \xx^2 = \frac {4 c_2^2} 3 \uu e^{\frac 8 3 \U} - 4 c_2 e^{c_1} \uu e^{2 \U} \\ 
			\Leftrightarrow & \frac d {ds} \left ( \xx^2 \right ) e^{- 4 \U} - 4 \uu e^{- 4 \U} \xx^2 = \frac {4 c_2^2} 3 \uu e^{- \frac 4 3 \U} - 4 c_2 e^{c_1} \uu e^{- 2 \U} \\
			\Leftrightarrow & \frac d {ds} \left ( \xx^2 e^{- 4 \U} + c_2^2 e^{- \frac 4 3 \U} - 2 c_2 e^{c_1} e^{- 2 \U} \right ) = 0, 	
		\end{align*}
		which means that
		\[
		\xx^2 e^{- 4 \U} + c_2^2 e^{- \frac 4 3 \U} - 2 c_2 e^{c_1} e^{- 2 \U}
		\]
		is a first integral. Therefore, 
		\begin{equation*}
			\xx^2 = 2 c_2 e^{c_1} e^{2 \U} - c_2^2 e^{\frac 8 3 \U} + c_3 e^{4 \U},
		\end{equation*}
		where $c_3 \in \mathbb R$.
		Since $\xx > 0$ at any point, we have
		\begin{equation*}
			\xx = \sqrt {2 c_2 e^{c_1} e^{2 \U} - c_2^2 e^{\frac 8 3 \U} + c_3 e^{4 \U}}.
		\end{equation*}
		Now we check if the last equation of \eqref{eq:differentialSystemTangentPartWithK} is satisfied, that is 
		\begin{align*}
			0 = & \yy (2 \ww - \yy) - \xx^2 \\
			  = & c_2 e^{\frac 4 3 \U} \left ( 2 e^{c_1} e^{\frac 2 3 \U} - c_2 e^{\frac 4 3 \U} \right ) - 2 c_2 e^{c_1} e^{2 \U} + c_2^2 e^{\frac 8 3 \U} - c_3 e^{4 \U} \\
			  = & - c_3 e^{4 \U},
		\end{align*}
		which is equivalent to 
		\begin{equation*}
			c_3 = 0.
		\end{equation*}
		Therefore, we obtain
		\begin{equation*} \label{eq:plusX}
			\xx = e^\U \sqrt {2 c_2 e^{c_1} - c_2^2 e^{\frac 2 3 \U}}.
		\end{equation*}
		In this case $c_2$ must be positive.
		
		From the second equation of \eqref{eq:differentialSystemTangentPartWithK} we deduce the following expression of $\vv$
		\begin{equation*}
			\vv = - \frac {2 \uu e^{- \frac 1 3 \U} \left ( 3 e^{c_1} - c_2 e^{\frac 2 3 \U} \right )} {3 \sqrt {2 c_2 e^{c_1} - c_2^2 e^{\frac 2 3 \U}}}.
		\end{equation*}
		Therefore, we have 
		\begin{equation} \label{eq:vInTheCaseKEqualsEpsilon}
			\vv = \displaystyle - \frac {2 \uu \left ( 3e^{c_1} - c_2 e^{\frac 2 3 \U} \right )} {3 e^{\frac 1 3 \U} \sqrt {2 c_2 e^{c_1} - c_2^2 e^{\frac 2 3 \U}}}
		\end{equation}
		and
		\begin{equation} \label{eq:solutionInTheCaseKEqualsEpsilon}
			\left \{
			\begin{array}{l}
				\ww = e^{c_1} e^{\frac 2 3 \U} \\[5pt]
				\xx = e^{\U} \sqrt {2 c_2 e^{c_1} - c_2^2 e^{\frac 2 3 \U}} \\[5pt]
				\yy = c_2 e^{\frac 4 3 \U} 
			\end{array}
			\right .,
		\end{equation}
		where $c_1 \in \mathbb R$ and $c_2 > 0$.
		
		As in the case of the particular solution presented in Theorem \ref{th:exampleYZero}, it is convenient to write the solution provided in Theorem \ref{th:exampleKEqualsEpsilon} in terms of the mean curvature $f = f(s)$. We know from Theorem \ref{th:existenceTheorem} that the function $\ww$ represents the mean curvature of $M$, thus
		\[
		f = f (s) = \ww (s) = e^{c_1} e^{\frac 2 3 \U (s)} > 0.
		\]
		First, we differentiate $f$ and obtain
		\begin{equation*}
			\dot f = \frac 2 3 e^{c_1} \uu e^{\frac 2 3 \U} = \frac 2 3 \uu f > 0,
		\end{equation*}
		which is equivalent to 
		\begin{equation*}
			\uu = \frac 3 2 \frac {\dot f} f.
		\end{equation*}
		Using the fact that $e^\U = e^{- \frac 3 2 c_1} f^{\frac 3 2}$ and putting $c = c_2 e^{- 2 c_1} > 0$, we obtain \eqref{eq:vSolutionKEqualsEpsInTermsOfF} and \eqref{eq:solutionKEqualsEpsInTermsOfF}.
		
		Now, since $(\uu, \ww, \xx, \yy, \kk)$ is a solution of \eqref{eq:differentialSystemTangentPartWithK}, we deduce that $f$ is a solution of the following second order ODE \eqref{eq:secondOrderODEForF}.
	\end{proof}

	\begin{remark}
		Formulas \eqref{eq:vInTheCaseKEqualsEpsilon} and \eqref{eq:solutionInTheCaseKEqualsEpsilon} describe the surfaces given in Theorem \ref{th:exampleKEqualsEpsilon} in terms of an arbitrarily fixed primitive $\U$ of $\uu$. Integrating \eqref{eq:uSolutionKEquals0Explicit}, \eqref{eq:uSolutionKEquals1Explicit}, \eqref{eq:uSolutionKEquals-1Explicit1} and \eqref{eq:uSolutionKEquals-1Explicit2}, we can find explicit formulas for the functions $\vv$, $\ww$, $\xx$ and $\yy$ given in \eqref{eq:vSolutionKEqualsEpsInTermsOfF} and \eqref{eq:solutionKEqualsEpsInTermsOfF}.
	\end{remark}	
	
	We note that, a first integral of \eqref{eq:secondOrderODEForF} is given by
	\begin{equation} \label{eq:firstOrderODEForF}
		\dot f^2 = f^2 \left ( C f^3 - \frac 4 9 \epsilon \right ),
	\end{equation}
	where $C \in \mathbb R$, if $\epsilon < 0$ and $C > 0$, if $\epsilon \geq 0$.

	Similarly to the case of Theorem \ref{th:exampleYZero}, we find that the metric $g$ is given by
	\begin{equation*}
		g (s, t) = ds^2 + f^{-3} (s) dt^2, \quad \text{for any } (s, t).
	\end{equation*}
	
	Now, in order to check that the non-CMC biconservative W-surfaces with flat normal bundle which have $K = \epsilon$ given in Theorem \ref{th:exampleKEqualsEpsilon} form a $1$-parameter family, we perform the change of coordinates $(s, t) \to (f = f(s), t)$. Using \eqref{eq:firstOrderODEForF}, we deduce that 
	\[
	df^2 = f^2 \left ( C f^3 - \frac 4 9 \epsilon \right ) ds^2
	\]
	and thus
	\begin{equation} \label{eq:metricKEqualsEpsilonInFT}
		g (f, t) = \frac 9 {f^2 \left ( 9 C f^3 - 4 \epsilon \right )} df^2 + f^{-3} dt^2, \quad \text{for any } (f, t).
	\end{equation}
	In this coordinates system, the functions $\uu$, $\vv$, $\ww$, $\xx$ and $\yy$ are given by
	\begin{equation} \label{eq:vInTheCaseKEqualsEpsInTheChartFT}
		\vv = - \frac {(3 - c f) \sqrt {9 C f^3 - 4 \epsilon}} {3 \sqrt c f^{\frac 1 2} \sqrt {2 - c f}}
	\end{equation}
	and
	\begin{equation} \label{eq:solutionInTheCaseKEqualsEpsInTheChartFT}
		\left \{
		\begin{array}{l}
			\displaystyle \uu = \frac 1 2 \sqrt {9 C f^3 - 4 \epsilon} \\[5pt]
			\displaystyle \ww = f \\[5pt]
			\displaystyle \xx = \sqrt c f^{\frac 3 2} \sqrt {2 - c f} \\[5pt]
			\displaystyle \yy = c f^2
		\end{array}
		\right .,
	\end{equation}
	where $c > 0$ and $f$ is positive.
	
	Since $K = \epsilon$, the constant $C$ which appears in \eqref{eq:metricKEqualsEpsilonInFT} is not an indexing constant (we can perform another change of coordinates such that the constant $C$ does not appear in the expression of the metric $g$). Thus, we have a $1$-parameter family of non-CMC and non-PNMC biconservative W-immersions with flat normal bundle indexed by $c$.
	
	We note that, from the definition of $\bar \Omega$, see \eqref{eq:differentialSystemTangentPartWithK}, the solution must satisfy $3 \ww - \yy \neq 0$ and $3 \ww^2 + \kk - \epsilon \neq 0$ at any point. Except for at most one point, these relations are satisfied. Thus, eventually restricting the domain interval $I$, the two inequalities are satisfied at any point.
	
	In the following, we present a solution of system \eqref{eq:differentialSystemTangentPartWithK} which generalizes those obtained in Theorems \ref{th:exampleYZero} and \ref{th:exampleKEqualsEpsilon}. The key point is to notice that in both cases the solutions satisfy 
	\[
	\frac {\dot \ww} \ww = \frac 2 3 \uu.
	\]
	
	\begin{proposition} \label{th:exampleMoreGeneral}
		Let $\varphi : M^2 \to N^4 (\epsilon)$ be a non-CMC biconservative W-surface with flat normal bundle. Assume that $3 f^2 + K - \epsilon \neq 0$, $\nabla^\perp E_3 \neq 0$, $A_4 \neq 0$ and $\nabla_{E_2} E_2 \neq 0$ at any point. Then, $M$ satisfies 
		\[
		\nabla _{E_2} E_2 = \frac 3 2 \frac {\dot f} f E_1
		\] 
		if and only if, locally, the functions $\uu$, $\vv$, $\ww$, $\xx$, $\yy$ and $\kk$ are given by one of the following
		\begin{enumerate} [label = \alph*)]
			\item \label{item:aGeneralization} 
			\begin{equation*}
				\vv = \frac {\dot f \left ( c_1^{-1} c_2 f - 3 \right )} {c_1^{- \frac 1 2} f^{\frac 3 2} \sqrt {2 c_2 - c_1^{-1} c_2^2 f - c_1^{-5} f^3}}
			\end{equation*}
			and
			\begin{equation*}
				\left \{
				\begin{array}{l}
					\displaystyle \uu = \frac 3 2 \frac {\dot f} f \\[10pt]
					\ww = f \\[5pt]
					\xx = c_1^{- \frac 1 2} f^{\frac 3 2} \sqrt {2 c_2 - c_1^{-1} c_2^2 f - c_1^{-5} f^3} \\[5pt]
					\yy = c_1^{-1} c_2 f^2 \\[5pt]
					\kk = \epsilon + c_1^{-6} f^6 
				\end{array}
				\right .,
			\end{equation*}
			where $c_1 > 0$, $c_2 > 0$ and $f$ is a positive solution of
			\begin{equation} \label{eq:secondOrderODEGeneralizationA}
				\ddot f f - \frac 5 2 \dot f^2 - \frac 2 3 \epsilon f^2 - \frac 2 3 c_1^{-6} f^8 = 0
			\end{equation} 
			such that $\dot f > 0$;
			
			\item \label{item:bGeneralization} relations \eqref{eq:vSolutionKEqualsEpsInTermsOfF}, \eqref{eq:solutionKEqualsEpsInTermsOfF} and 
			\begin{equation*}
				\kk = \epsilon;
			\end{equation*}
		
			\item \label{item:cGeneralization} 
			\begin{equation*}
				\vv = \frac {\dot f \left ( c_1^{-1} c_2 f - 3 \right )} {c_1^{- \frac 1 2} f^{\frac 3 2} \sqrt {2 c_2 - c_1^{-1} c_2^2 f + c_1^{-5} f^3}}
			\end{equation*}
			and
			\begin{equation*}
				\left \{
				\begin{array}{l}
					\displaystyle \uu = \frac 3 2 \frac {\dot f} f \\[10pt]
					\ww = f \\[5pt]
					\xx = c_1^{- \frac 1 2} f^{\frac 3 2} \sqrt {2 c_2 - c_1^{-1} c_2^2 f + c_1^{-5} f^3} \\[5pt]
					\yy = c_1^{-1} c_2 f^2 \\[5pt]
					\kk = \epsilon - c_1^{-6} f^6 
				\end{array}
				\right .,
			\end{equation*}
			where $c_1 > 0$, $c_2 \in \mathbb R$ and $f$ is a positive solution of 
			\begin{equation*} 
				\ddot f f - \frac 5 2 \dot f^2 - \frac 2 3 \epsilon f^2 + \frac 2 3 c_1^{-6} f^8 = 0
			\end{equation*} 
			such that $\dot f > 0$.
		\end{enumerate}
	\end{proposition}

	As in the previous cases, we make the change of coordinates $(s, t) \to (f = f(s), t)$. We write here only \ref{item:aGeneralization} of Proposition \ref{th:exampleMoreGeneral} in terms of $f$, the item \ref{item:cGeneralization} can be treated analogously. 
	
	Taking into account the fact that $f$ must satisfy the second order ODE \eqref{eq:secondOrderODEGeneralizationA} with a first integral 
	\begin{equation*}
		\dot f^2 = \frac 4 9 f^2 \left ( C f^3 + c_1^{-6} f^6 - \epsilon  \right ),
	\end{equation*}
	we obtain that the metric $g$ is given by 
	\begin{equation*}
		g (f, t) = \frac 9 {4 f^2 \left ( C f^3 + c_1^{-6} f^6 - \epsilon  \right )} df^2 + f^{-3} dt^2, \quad \text{for any } (f, t),
	\end{equation*}
	where $C \in \mathbb R$.
	
	The solution from \ref{item:aGeneralization} of Proposition \ref{th:exampleMoreGeneral} can be written as
	\begin{equation*}
		\vv = \frac {2 \left ( c_1^{-1} c_2 f - 3 \right ) \sqrt {C f^3 + c_1^{-6} f^6 - \epsilon}} {3 c_1^{- \frac 1 2} f^{\frac 1 2}  \sqrt {2 c_2 - c_1^{-1} c_2^2 f - c_1^{-5} f^3}}
	\end{equation*}
	and
	\begin{equation*}
		\left \{
		\begin{array}{l}
			\displaystyle \uu = \sqrt {C f^3 + c_1^{-6} f^6 - \epsilon} \\[10pt]
			\ww = f \\[5pt]
			\xx = c_1^{- \frac 1 2} f^{\frac 3 2} \sqrt {2 c_2 - c_1^{-1} c_2^2 f - c_1^{-5} f^3} \\[5pt]
			\yy = c_1^{-1} c_2 f^2 \\[5pt]
			\kk = \epsilon + c_1^{-6} f^6 
		\end{array}
		\right .,
	\end{equation*}
	where $c_1 > 0$, $c_2 > 0$ and $f$ is a positive.
	
	From the definition of $\bar \Omega$, the solution must satisfy $3 \ww - \yy \neq 0$ and $3 \ww^2 + \kk - \epsilon \neq 0$ at any point. Eventually, except for at most two points,  these relations are satisfied.
	
	We note that, excepting the case \ref{item:bGeneralization} of Proposition \ref{th:exampleMoreGeneral}, the solutions form a $3$-parametric family of non-CMC biconservative W-surfaces with flat normal bundle. If we fix the domain metric $g$, we have a $2$-parameter family of non-CMC and non-PNMC biconservative W-immersions with flat normal bundle indexed by $c_1$ and $c_2$.
	
	\begin{remark}
		If we choose $c_2 = 0$ in \ref{item:cGeneralization} of Proposition \ref{th:exampleMoreGeneral}, we obtain the result in Theorem \ref{th:exampleYZero}.
	\end{remark}
	
	At the end of this section, we remark that in the proof of Theorem \ref{th:existenceTheorem} the relation $3 f^2 + K - \epsilon \neq 0$ was not needed, even if it was implicitly ensured by the definition of the domain $\Omega$. In fact, if we assume the equality 
	\begin{equation} \label{eq:equalityCondition}
		3 f^2 + K - \epsilon = 0
	\end{equation}
	and slightly modify $\Omega$, the Theorem \ref{th:existenceTheorem} remains valid and we have an existence result. Moreover, in the following result we determine all non-CMC biconservative W-surfaces with flat normal bundle which satisfy \eqref{eq:equalityCondition}.
	
	\begin{theorem} \label{th:exampleKEquals3F2+epsilon}
		Let $\varphi : M^2 \to N^4 (\epsilon)$ be a non-CMC biconservative W-surface with flat normal bundle. Assume that $\left \langle \nabla ^\perp _{E_1} E_3, E_4 \right \rangle \left \langle \nabla ^\perp _{E_2} E_3, E_4 \right \rangle = 0$ on $M$ and $\nabla^\perp E_3 \neq 0$, $A_4 \neq 0$ and $\nabla_{E_2} E_2 \neq 0$ at any point. Then, $M$ satisfies $3 f^2 + K - \epsilon = 0$ if and only if, locally,
		\begin{equation*}
			\vv = \frac {c \dot f} {f \sqrt {- 4 c f^{\frac 1 2} - c^2}}
		\end{equation*}
		and 
		\begin{equation*}
			\left \{
			\begin{array}{l}
				\displaystyle \uu = \frac 3 4 \frac {\dot f} f \\[10pt]
				\ww = f \\[5pt]
				\xx = \sqrt {- 4 c f^{\frac 3 2} - c^2 f} \\[5pt]
				\yy = 3f + c f^{\frac 1 2} \\[5pt]
				\kk = \epsilon - 3 f^2
			\end{array}
			\right .,
		\end{equation*}
		where $c < 0$ and $f$ is a positive solution of 
		\begin{equation} \label{eq:secondOrderODEKEqualsEpsilon-3F2}
			\ddot f f - \frac 7 4 \dot f^2 + 4 f^4 - \frac 4 3 \epsilon f^2 = 0,
		\end{equation}
		such that $\dot f > 0$.	
	\end{theorem}
	
	Taking into account that a first integral of \eqref{eq:secondOrderODEKEqualsEpsilon-3F2} is given by
	\begin{equation*}
		\dot f^2 = f^2 \left ( C f^{\frac 3 2} - 16 f^2 - \frac {16} 9 \epsilon \right ),
	\end{equation*}
	we obtain that the metric
	\begin{equation} \label{eq:metricExampleKEqualsEps-3F2}
		g(s, t) = ds^2 + f^{- \frac 3 2} (s) dt^2, \quad \text{for any } (s, t)
	\end{equation} 
	can also be given by 
	\begin{equation} \label{eq:metricExampleKEqualsEps-3F2InFT}
		g (f, t) = \frac 1 {f^2 \left ( C f^{\frac 3 2} - 16 f^2 - \frac {16} 9 \epsilon \right )} df^2 + f^{- \frac 3 2} dt^2, \quad \text{for any } (f, t),
	\end{equation}
	where $C \in \mathbb R$, if $\epsilon < 0$ and $C > 0$, if $\epsilon \geq 0$.
	
	The solution of Theorem \ref{th:exampleKEquals3F2+epsilon} can be written as
	\begin{equation*}
		\vv = \frac {c \sqrt{C f^{\frac 3 2} - 16 f^2 - \frac {16} 9 \epsilon}} {\sqrt{- 4 c f^{\frac 1 2} - c^2}}
	\end{equation*}
	and
	\begin{equation*}
		\left \{
		\begin{array}{l}
			\uu = \frac 3 4 \sqrt {C f^{\frac 3 2} - 16 f^2 - \frac {16} 9 \epsilon} \\[5pt]
			\ww = f \\[5pt]
			\xx = \sqrt {- 4 c f^{\frac 3 2} - c^2 f} \\[5pt]
			\yy = 3 f + c f^{\frac 1 2} \\[5pt]
			\kk = \epsilon - 3 f^2
		\end{array}
		\right .
	\end{equation*}
	where $c < 0$ and $f$ is positive.
	
	From the definition of $\bar \Omega$, the solution must satisfy $3 \ww - \yy \neq 0$ at any point. In this case, this relation is always satisfied.
	
	One can check that the domain metrics given by \eqref{eq:metricExampleKEqualsEps-3F2InFT} represent a $1$-parametric family of non-isometric metrics indexed by the constant $C$. If we fix a metric $g$, that is we fix the parameter $C$, we have a $1$-parameter family of non-CMC and non-PNMC biconservative W-immersions with flat normal bundle indexed by $c$.
	
	\begin{remark}
		In our approach for classifying non-CMC biconservative W-surfaces with flat normal bundle it was essential to have $b_1 b_2 = 0$ on $M$, that is $\left \langle \nabla ^\perp _{E_1} E_3, E_4 \right \rangle \left \langle \nabla ^\perp _{E_2} E_3, E_4 \right \rangle = 0$, as this condition lead us to the main system \eqref{eq:differentialSystemTangentPartConverse}. The case $b_1 b_2 \neq 0$, which implies $3 f^2 + K - \epsilon = 0$, remains uncovered by this paper.
	\end{remark}

	We recall that in the case of non-CMC biconservative surfaces $\psi : \left ( M^2, g \right ) \to N^3 (\epsilon)$ the Gauss formula is equivalent to $K = \epsilon - 3 \left ( f^\psi \right )^2$ (see \cite{CaddeoMontaldoOniciucPiu2014}). Inspired by this we have the following result.
	
	\begin{theorem} \label{th:twoImmersionsKEqualsEps-3F2}
		Let $\left ( M^2 = I \times \mathbb R, g \right )$ be the abstract surface with the metric given by \eqref{eq:metricExampleKEqualsEps-3F2} with $f$ a positive function satisfying $\dot f > 0$ and \eqref{eq:secondOrderODEKEqualsEpsilon-3F2}. Then,
		\begin{enumerate} [label = \alph*)]
			\item \label{item:propositionDifferentImmersionsA} there exists a $1$-parametric family of non-CMC biconservative W-surfaces $\varphi : \left ( M^2, g \right ) \to N^4 (\epsilon)$ with flat normal bundle satisfying $\left \langle \nabla ^\perp _{E_1} E_3, E_4 \right \rangle \left \langle \nabla ^\perp _{E_2} E_3, E_4 \right \rangle = 0$, $K = \epsilon - 3 \left ( f^\varphi \right )^2$ on $M$ and $\nabla^\perp E_3 \neq 0$, $A_4 \neq 0$, $\nabla_{E_2} E_2 \neq 0$ at any point. Each of these surfaces in this family have the same mean curvature given by
			\begin{equation*}
				\left ( f^\varphi \right ) ^2 = \frac 1 3 (\epsilon - K);
			\end{equation*}
			
			\item \label{item:propositionDifferentImmersionsB} there exists a unique non-CMC biconservative surface $\psi : \left ( M^2, g \right ) \to N^3 (\epsilon)$. Its mean curvature is
			\begin{equation*}
				\left ( f^\psi \right )^2 = \left ( f^\varphi \right ) ^2 = \frac 1 3 (\epsilon - K).
			\end{equation*}
		\end{enumerate}
	\end{theorem}

	\begin{proof}
		For item \ref{item:propositionDifferentImmersionsA} we have seen in Theorem \ref{th:exampleKEquals3F2+epsilon} that such abstract surface determines a $1$-parametric family of immersions satisfying the desired conditions. For item \ref{item:propositionDifferentImmersionsB} it is enough to notice that the metric $g$ given in \eqref{eq:metricExampleKEqualsEps-3F2} is precisely the metric given in Remark 4.3 of \cite{FetcuNistorOniciuc2016} and then apply Theorem 4.5 of \cite{FetcuNistorOniciuc2016}.
	\end{proof}	

	\begin{remark} \label{th:remarkMetric}
		We note that the metric $g$ given in \eqref{eq:metricExampleKEqualsEps-3F2} can be written as 
		\begin{equation*}
			g (s, t) = ds^2 + (\epsilon - K)^{- \frac 3 4} dt^2, \quad \text{for any } (s, t),
		\end{equation*}
		where $K = K(s)$ satisfies $\epsilon - K > 0$, $|\grad K|^2 = \dot K^2 > 0$ and 
		\begin{equation*}
			24 (\epsilon - K) \ddot K + 33 \dot K^2 + 64 K (\epsilon - K)^2 = 0.
		\end{equation*}
		The shape operators of an immersion $\varphi$ from item \ref{item:propositionDifferentImmersionsA} of Proposition \ref{th:twoImmersionsKEqualsEps-3F2} are given by
		\begin{equation*}
			A^\varphi _3 = 
			\begin{pmatrix}
				- 3^{- \frac 1 2} (\epsilon - K) ^{\frac 1 2} - c 3^{- \frac 1 4} (\epsilon - K)^{\frac 1 4} & 0 \\
				0 & 3^{\frac 1 2} (\epsilon - K) ^{\frac 1 2} + c 3^{- \frac 1 4} (\epsilon - K)^{\frac 1 4}
			\end{pmatrix}
		\end{equation*}
		and
		\begin{equation*}
			A^\varphi _4 = 
			\begin{pmatrix}
				\sqrt {-4 c 3 ^{-\frac 3 4} (\epsilon - K)^{\frac 3 4} - c^2 3^{- \frac 1 2} (\epsilon - K)^{\frac 1 2}} & 0 \\
				0 & - \sqrt {-4 c 3 ^{-\frac 3 4} (\epsilon - K)^{\frac 3 4} - c^2 3^{- \frac 1 2} (\epsilon - K)^{\frac 1 2}}
			\end{pmatrix},
		\end{equation*}
		where $c < 0$.
		
		The shape operator of the immersion $\psi$ from item \ref{item:propositionDifferentImmersionsB} of Proposition \ref{th:twoImmersionsKEqualsEps-3F2} is given by
		\begin{equation*}
			A^\psi = 
			\begin{pmatrix}
				- 3^{- \frac 1 2} (\epsilon - K)^{\frac 1 2} & 0 \\
				0 & 3^{\frac 1 2} (\epsilon - K)^{\frac 1 2}
			\end{pmatrix}.
		\end{equation*}
	\end{remark}
	
	\subsection{The PNMC case - a different approach} \label{sec:PNMC}
	
	The PNMC case can be seen as a singular case of \eqref{eq:differentialSystemTangentPartConverse}. Recall that a surface is PNMC, that is $\nabla ^\perp E_3 = 0$, if and only if $b_1 = b_2 = 0$. When the surface $M$ is PNMC, from \eqref{eq:biharmonicEquationTangentPart}, we have 
	\[
	E_1 = \frac {\grad f} {|\grad f|} \quad \text{and} \quad k_1 = - f.
	\] 
	Thus, $0 = k_1 + f = 3f - k_2$ on $M$ and, since $b_1 = b_2 = 0$, now \eqref{eq:E1MeanCurvature} is trivially satisfied and gives no information. Consequently, the second equation of \eqref{eq:differentialSystemTangentPartConverse} will not appear in the new system. Further, analyzing Propositions \ref{th:propertiesOurSurface} and \ref{th:propertiesBiconservativeSurface}, we obtain
	\begin{theorem}
		Let $\varphi: M^2 \to N^4(\epsilon)$ be a non-CMC, PNMC biconservative surface. Then, around any point, there exist local coordinates $(s, t)$ such that $f = f(s)$, $k_1 = k_1 (s)$, $k_2 = k_2 (s)$, $\alpha = \alpha (s)$, $a_2 = a_2 (s)$, $b_1 = b_1 (s)$ and $K = K (s)$. Moreover, the tuple $\left ( \uu, \xx, \yy \right ) = (a_2, \alpha, k_2)$ is a solution of the following first order ODE system
		\begin{equation} \label{eq:differentialSystemPNMC}
			\left \{
			\begin{array}{l}
				\displaystyle \dot \uu = \epsilon - \frac 1 3 \yy^2 - \xx^2 + \uu^2 \\[10pt]
				\displaystyle \dot \xx = 2 \xx \uu \\[5pt]
				\displaystyle \dot \yy = \frac 4 3 \yy \uu
			\end{array} 
			\right .,
		\end{equation}
		where $\dot \uu$, $\dot \xx$ and $\dot \yy$ represent the derivatives with respect to $s$ of $\uu$, $\xx$ and $\yy$, respectively and we can assume
		\[
		\uu > 0, \quad \xx > 0, \quad \yy > 0 \quad \text{and} \quad \dot \yy > 0.
		\]
	\end{theorem}

	It was essentially proved in \cite{NistorOniciucTurgaySen2023} and \cite{NistorRusu2024} that system \eqref{eq:differentialSystemPNMC} represents the compatibility conditions of this PNMC biconservative surface problem, that is the analog of Theorem \ref{th:existenceTheorem} holds. 
	
	In this singular case there are two properties which do not hold in the non-PNMC case. First, if we fix the abstract surface $\left ( M^2, g \right )$, that is we fix $\uu$, and if there exists a PNMC biconservative immersion $\varphi : \left ( M^2, g \right ) \to N^4 (\epsilon)$, then it has to be unique, as shown in the following result.
	
	\begin{theorem}[\cite{NistorOniciucTurgaySen2023}, \cite{NistorRusu2024}, \cite{SenTurgay2018}] \label{th:uniquenessPNMC}
		If an abstract surface $\left ( M^2, g \right )$ admits two non-CMC, PNMC biconservative immersions in $N^4 (\epsilon)$, then these immersions differ by an isometry of $N^4 (\epsilon)$.		
	\end{theorem}

	\begin{proof} 
		Using \eqref{eq:differentialSystemPNMC}, we can provide a simpler proof than the one presented in \cite{NistorOniciucTurgaySen2023}, \cite{NistorRusu2024} and \cite{SenTurgay2018}. 
		
		Since we fixed the abstract surface $\left ( M^2, g \right )$, we fixed $\uu$. Let $\U$ be an arbitrarily fixed primitive of $\uu$. Using the fact that $\xx > 0$ and the second equation of \eqref{eq:differentialSystemPNMC} we obtain that 
		\begin{align*}
			\frac {\dot \xx} \xx = 2 \uu \Leftrightarrow \ln \xx = 2 \U + c_1 \Leftrightarrow \xx = e^{c_1} e^{2 \U}.
		\end{align*}
		Redenoting $c_1 = e^{c_1}$, we obtain that 
		\begin{equation*}
			\xx = c_1 e^{2 \U},
		\end{equation*}
		where $c_1$ is a positive real constant.
		
		Similarly, since $\yy > 0$, the third equation of \eqref{eq:differentialSystemPNMC} implies that
		\begin{equation*}
			\yy = c_2 e^{\frac 4 3 \U},
		\end{equation*}
		where $c_2$ is a positive real constant.
		
		Since $\U$ is a fixed primitive of $\uu$, we deduce that $\xx$ and $\yy$ are uniquely determined by $c_1$ and $c_2$. In the following we show that $c_1 $ and $c_2$ are uniquely determined by $\uu$ and $\U$.
		
		Replacing the expressions of $\xx$ and $\yy$ in the first equation of \eqref{eq:differentialSystemPNMC}, we obtain 
		\begin{equation} \label{eq:equationC1C2PNMC1}
			e^{4 \U} c_1^2 + \frac 1 3 e^{\frac 8 3 \U} c_2^2 = \uu^2 + \epsilon - \dot \uu.
		\end{equation}
		Differentiating \eqref{eq:equationC1C2PNMC1}, using the second and third equations of \eqref{eq:differentialSystemPNMC} and dividing by $4 \uu$, we obtain 
		\begin{equation} \label{eq:equationC1C2PNMC2}
			e^{4 \U} c_1^2 + \frac 2 9 e^{\frac 8 3 \U} c_2^2 = \frac {\dot \uu} 2 - \frac {\ddot \uu} {4 \uu}.
		\end{equation}
	
		Subtracting \eqref{eq:equationC1C2PNMC2} from \eqref{eq:equationC1C2PNMC1}, we obtain 
		\begin{equation*}
			c_2^2 = 9 e^{- \frac 8 3 \U} \left ( \frac {\ddot \uu} {4 \uu} - \frac {3 \dot \uu} 2 + \uu^2 + \epsilon \right ).
		\end{equation*}
		Replacing this in \eqref{eq:equationC1C2PNMC1}, we obtain
		\begin{equation*}
			c_1^2 = e^{- 4 \U} \left ( \frac {7 \dot \uu} 2 - \frac {3 \ddot \uu} {4 \uu} - 2 \uu^2 - 2 \epsilon \right ).
		\end{equation*}
		Since $c_1 > 0$ and $c_2 >0$, we obtain that $c_1$ and $c_2$ are uniquely determined by $\uu$ and $\U$. Therefore, $\xx$ and $\yy$ are unique and the conclusion follows.
	\end{proof}

	Second, we want to determine all abstract surfaces $\left (M^2, g \right )$ which admit (unique) PNMC biconservative immersions. This was done in \cite{NistorOniciucTurgaySen2023}, \cite{NistorRusu2024} and \cite{SenTurgay2018} by geometric means, but here, taking into account that the metric $g$ is determined by the function $\uu$, we find the necessary and sufficient condition that $\uu$ must satisfy. 

	\begin{proposition} [\cite{NistorOniciucTurgaySen2023}] \label{th:existencePNMC}
		An abstract surface $\left ( M^2, g \right )$ admits (unique) non-CMC, PNMC biconservative immersions in $N^4 (\epsilon)$ if and only if the function $\uu$ satisfies the following third order ODE
		\begin{equation} \label{eq:compatibilityConditionForTheSystem}
			3 \dddot \uu \uu - 3 \ddot \uu \dot \uu + 72 \dot \uu \uu^3 - 26 \ddot \uu \uu^2 - 32 \epsilon \uu^3 - 32 \uu^5 = 0.
		\end{equation}
	\end{proposition}

	\begin{proof}
		For the direct implication, we consider system \eqref{eq:differentialSystemPNMC} and assume that the function $\uu$ is given. 
		
		First, we suppose that \eqref{eq:differentialSystemPNMC} with $\uu$ given has a solution $(\xx, \yy)$ and find the ODE that $\uu$ must satisfy.
		
		Since $\uu$ is smooth, there exists a positive smooth function $f$ such that $\dot f > 0$ and 
		\[
		\uu = \frac 3 4 \frac {\dot f} f.
		\]
		Note that $f$ is determined up to a multiplicative positive constant. In the following we arbitrarily fix such a $f$. Taking into account the second and third equations of \eqref{eq:differentialSystemPNMC}, we obtain that the general solution of the system, with $\uu$ given, is of the form
		\begin{equation*}
			\left \{
			\begin{array}{l}
				\displaystyle \xx = c_1 f^{\frac 3 2} \\[10pt]
				\displaystyle \yy = c_2 f
			\end{array}
			\right .,
		\end{equation*}
		for some positive real constants $c_1$ and $c_2$.
		
		If we redenote $c_2 f / 3$ by $f$ and put $c = c_1 \left ( 3 / c_2 \right )^{3/ 2} > 0$, we obtain
		\begin{equation*}
			\left \{
			\begin{array}{l}
				\displaystyle \xx = c f^{\frac 3 2} \\[10pt]
				\displaystyle \yy = 3 f
			\end{array}
			\right ..
		\end{equation*}
		Thus, taking into account these expressions of $\xx$ and $\yy$ and replacing in the first equation of \eqref{eq:differentialSystemPNMC}, we have
		\begin{align*}
			\dot \uu = & \epsilon - 3 f^2 - c^2 f^3 + \uu ^2 \\
			\Leftrightarrow c^2 = & \frac {\epsilon + \uu^2 - \dot \uu} {f^3} - \frac 3 f \\
			\Rightarrow 0 = & \frac {\left ( 2 \dot \uu \uu - \ddot \uu \right ) f^3 - 3 \dot f f^2 \left ( \epsilon + \uu^2 - \dot \uu \right )} {f^6} + \frac {3 \dot f} {f^2}.
		\end{align*}
		Multiplying this relation by $- f^3$ and taking into account that $\dot f = 4 f \uu / 3$, we obtain 
		\begin{equation} \label{eq:expressionFSquaredPNMC}
			\ddot \uu + 4 \epsilon \uu + 4 \uu^3 - 6 \dot \uu \uu - 4 \uu f^2 = 0.
		\end{equation}
		Differentiating \eqref{eq:expressionFSquaredPNMC} we get
		\begin{equation*}
			0 = \dddot \uu + 4 \epsilon \dot \uu + 12 \dot \uu \uu^2 - 6 \ddot \uu \uu - 6 \dot \uu^2 - 4 \left ( \dot \uu + \frac 8 3 \uu^2 \right ) f^2.
		\end{equation*}
		Replacing $f^2$ from \eqref{eq:expressionFSquaredPNMC} in the last relation, we get 
		\begin{align*}
			0 = & \dddot \uu + 4 \epsilon \dot \uu + 12 \dot \uu \uu^2 - 6 \ddot \uu \uu - 6 \dot \uu^2 - 4 \left ( \dot \uu + \frac 8 3 \uu^2 \right ) \frac {\ddot \uu + 4 \epsilon \uu + 4 \uu^3 - 6 \dot \uu \uu} {4 \uu} \\
			\Leftrightarrow 0 = & 3 \dddot \uu \uu + 12 \epsilon \dot \uu \uu + 36 \dot \uu \uu^3 - 18 \ddot \uu \uu^2 - 18 \dot \uu^2 \uu - \left ( 3 \dot \uu + 8 \uu^2 \right ) \left ( \ddot \uu + 4 \epsilon \uu + 4 \uu^3 - 6 \dot \uu \uu \right )
		\end{align*}
		which is equivalent to \eqref{eq:compatibilityConditionForTheSystem}.
		
		Conversely, we consider a solution $\uu$ of \eqref{eq:compatibilityConditionForTheSystem} and show that the system \eqref{eq:differentialSystemPNMC} associated to $\uu$ admits a solution $(\xx, \yy)$. 
		
		Again, we write $\uu$ as
		\[
		\uu = \frac 3 4 \frac {\dot f} f.
		\]
		We note that $f$ is determined up to multiplicative positive constants.
		
		From the second and third equations of \eqref{eq:differentialSystemPNMC} we find that, for given initial conditions $(s_0, \xx_0, \yy_0)$ there exist a unique smooth positive function $f$ and a unique positive real constant $c$ such that  
		\begin{equation*}
			\xx = c f^{\frac 3 2}, \quad \yy = 3 f \quad \text{and} \quad c f^{\frac 3 2} (s_0) = \xx_0, \quad 3 f (s_0) = \yy_0.
		\end{equation*} 
		Now we impose that the initial conditions $(s_0, \xx_0, \yy_0)$ satisfy the following two conditions 
		\begin{equation} \label{eq:initialConstraintsPNMC}
			\left \{
			\begin{array}{l}
				\displaystyle \xx_0^2 + \frac 1 3 \yy_0^2 = \epsilon + \uu^2 (s_0) - \dot \uu (s_0) \\[10pt]
				\displaystyle 6 \uu (s_0) \xx_0^2 + \frac {14} 9 \uu (s_0) \yy_0^2 = - \ddot \uu (s_0) + 2 \epsilon \uu (s_0) + 2 \uu^3 (s_0)
			\end{array}
			\right .
		\end{equation}
		and we prove that $(\xx, \yy)$ satisfies the first equation of \eqref{eq:differentialSystemPNMC}. For this, we denote 
		\begin{equation*}
			\beta = \dot \uu - \epsilon + \frac 1 3 \yy^2 + \xx^2 - \uu^2.
		\end{equation*}
		Following the same steps as in the first part of the proof and taking into account that \eqref{eq:compatibilityConditionForTheSystem} holds, we find that 
		\begin{equation} \label{eq:differentialEquationRest}
			3 \uu \ddot \beta - \left ( 3 \dot \uu + 20 \uu^2 \right ) \dot \beta + 32 \uu^3 \beta = 0.
		\end{equation}
		From \eqref{eq:initialConstraintsPNMC} we obtain that $\beta (s_0) = \dot \beta (s_0) = 0$ and, taking into account \eqref{eq:differentialEquationRest}, we find out that \eqref{eq:differentialSystemPNMC} is satisfied.
	\end{proof}
	
	\begin{remark}
		Relation \eqref{eq:compatibilityConditionForTheSystem} can be seen as the compatibility condition for the system \eqref{eq:differentialSystemPNMC} with $\uu$ given.
	\end{remark}
	
	\section{Biharmonic surfaces}
	
	In this section we provide a characterization of biharmonic W-surfaces with flat normal bundle and we show that the surfaces presented in Theorems \ref{th:exampleYZero}, \ref{th:exampleKEqualsEpsilon}, \ref{th:exampleKEquals3F2+epsilon} and Proposition \ref{th:exampleMoreGeneral} cannot be biharmonic.
	
	\begin{theorem} 
		Let $\varphi: M^2 \to N^4(\epsilon)$ be a non-CMC W-surface with flat normal bundle. Assume that $3 f^2 + K - \epsilon \neq 0$, $\nabla^\perp E_3 \neq 0$, $A_4 \neq 0$ and $\nabla_{E_2} E_2 \neq 0$ at any point. Then, $M$ is biharmonic if and only if, around any point, there exist local coordinates $(s, t)$ such that $f = f(s)$, $k_1 = k_1 (s)$, $k_2 = k_2 (s)$, $\alpha = \alpha (s)$, $a_2 = a_2 (s)$, $b_1 = b_1 (s)$, $K = K (s)$ and the following first order ODE system must be satisfied
		\begin{equation} \label{eq:differentialSystemBiharmonicity1}
			\left \{
			\begin{array}{l}
				\dot a_2 = \epsilon + k_2 (2 f - k_2) - \alpha^2 + a_2^2 \\[5pt]
				\displaystyle (3f - k_2) \dot f = - f \alpha b_1 \\[5pt]
				\dot \alpha = 2 \alpha a_2 + k_2 b_1 \\[5pt]
				\dot k_2 = 2 a_2 (k_2 - f) - \alpha b_1 \\[5pt]
				\ddot f = a_2 \dot f + f \left ( b_1^2 + k_1^2 + k_2^2 - 2 \epsilon \right ) \\[5pt]
				\displaystyle \dot b_1 = - \frac {2 b_1} f \dot f + a_2 b_1 - \alpha (k_2 - k_1)
			\end{array} 
			\right .,
		\end{equation}
		where $\dot a_2$, $\dot f$, $\dot \alpha$, $\dot k_2$, $\dot b_1$ represent the derivatives with respect to $s$ of $a_2$, $f$, $\alpha$, $k_2$ and $b_1$, respectively.
	\end{theorem}
	
	\begin{proof}
		Recall that any biharmonic surface is biconservative, so the first four equations of the system are the equations derived from the tangent component of the biharmonic equation, that is system \eqref{eq:differentialSystemTangentPart}. In the following, we deduce the last two equations of the system from the normal component of the biharmonic equation \eqref{eq:biharmonicEquationNormalPart}.
		
		First, using \eqref{eq:ConnectionFormSurface} and \eqref{eq:ConnectionFormNormal}, we compute
		\begin{align*}
			\Delta ^\perp H = & \Delta ^\perp (f E_3) = - \left ( \nabla ^\perp _{E_1} \nabla ^\perp _{E_1} (f E_3) - \nabla ^\perp _{\nabla _{E_1} E_1} (f E_3) + \nabla ^\perp _{E_2} \nabla ^\perp _{E_2} (f E_3) - \nabla ^\perp _{\nabla _{E_2} E_2} (f E_3) \right ) \\
							= & - E_1 (E_1 (f)) E_3 - E_1 (f) \nabla ^\perp _{E_1} E_3 - E_1 (f) b_1 E_4 - f E_1 (b_1) E_4 - f b_1 \nabla ^\perp _{E_1} E_4 \\
							  & - a_1 E_2 (f) E_3 - a_1 f \nabla ^\perp _{E_2} E_3 - E_2 (E_2 (f)) - E_2 (f) \nabla ^\perp _{E_2} E_3 - E_2 (f) b_2 E_4 \\
							  & - f E_2 (b_2) E_4 - f b_2 \nabla ^\perp _{E_2} E_4 + a_2 E_1 (f) E_3 + a_2 f \nabla ^\perp _{E_1} E_3.
		\end{align*}
		Therefore, 
		\begin{align*}
			\Delta ^\perp H = & \Bigl ( - E_1 (E_1 (f)) - E_2 (E_2 (f)) + f \left ( b_1^2 + b_2^2 \right ) - a_1 E_2 (f) + a_2 E_1 (f) \Bigr ) E_3 \\
							  & + \Bigl ( f (b_1 a_2 - b_2 a_1) - 2 \bigl ( b_1 E_1 (f) + b_2 E_2 (f) \bigr ) - f \bigl ( E_1 (b_1) + E_2 (b_2) \bigr ) \Bigr ) E_4. \notag
		\end{align*}
		Using \eqref{eq:SecondFundamentalForm}, we obtain 
		\begin{equation*}
			\trace B (A_H (\cdot), \cdot) = f \left ( k_1^2 + k_2^2 \right ) E_3 - f \alpha (k_2 - k_1) E_4.
		\end{equation*}
		Therefore, \eqref{eq:biharmonicEquationNormalPart} is equivalent to 
		\begin{equation*}
			\left \{
			\begin{array}{l}
				- E_1 (E_1 (f)) - E_2 (E_2 (f)) + f \left ( b_1^2 + b_2^2 + k_1^2 + k_2^2 - 2 \epsilon \right ) - a_1 E_2 (f) + a_2 E_1 (f) = 0 \\[5pt]
				2 \bigl ( b_1 E_1 (f) + b_2 E_2 (f) \bigr ) + f \bigl ( E_1 (b_1) + E_2 (b_2) \bigr ) + f (a_1 b_2 - a_2 b_1) + f \alpha (k_2 - k_1) = 0
			\end{array}
			\right ..
		\end{equation*}
		Taking into account Lemma \ref{th:constraintNormalConnection}, \eqref{eq:alphaNeq0E2fZero}, \eqref{eq:a1Zero} and \eqref{eq:VanishingOfTheDerivativesInTheDirOfE2}, the conclusion follows.
	\end{proof}

	As in the biconservative case, we denote $\uu = \uu (s) = a_2 (s)$, $\vv = \vv (s) = b_1 (s)$, $\ww = \ww (s) = f (s)$, $\xx = \xx (s) = \alpha (s)$, $\yy = \yy (s) = k_2 (s)$ and $\zz = \zz (s) = \dot f (s)$ and consider 
	\[
	F : \mathbb R^* \times \mathbb R^* \times (0, \infty) \times (0, \infty) \times \mathbb R \times (0, \infty) \to \mathbb R^6
	\]
	defined by
	\begin{equation*}
		F (\uu, \vv, \ww, \xx, \yy, \zz) = 
		\begin{pmatrix}
			\epsilon + \yy ( 2 \ww - \yy) - \xx^2 + \uu^2 \\[5pt]
			\displaystyle - \frac {2 \vv \zz} \ww + \uu \vv + 2 \xx (\ww - \yy) \\[10pt]
			\zz \\[5pt]
			2 \xx \uu + \yy \vv \\[5pt]
			2 \uu (\yy - \ww) - \xx \vv \\[5pt]
			\uu \zz + \ww \left ( \vv^2 + (2 \ww - \yy)^2 + \yy^2 - 2 \epsilon \right )
		\end{pmatrix}.
	\end{equation*}
	Then, system \eqref{eq:differentialSystemBiharmonicity1} is equivalent to the following differential system with a constraint 
	\begin{equation} \label{eq:differentialSystemBiharmonicity}
		\left \{
		\begin{array}{l}
			\dot X (s) = F (X (s)) \\[5pt]
			(3 \ww (s) - \yy (s)) \zz (s) = - \ww (s) \xx (s) \vv (s)
		\end{array}
		\right ., \quad \text{for any } s,
	\end{equation}
	where $X (s) = (\uu (s), \vv (s), \ww (s), \xx (s), \yy (s), \zz (s))$.
	
	We note that the constraint of system \eqref{eq:differentialSystemBiharmonicity} will, presumably, prevent the existence of biharmonic $W$-surfaces with flat normal bundle.
	
	In the following, we show that the biconservative surfaces provided in Theorems \ref{th:exampleYZero}, \ref{th:exampleKEqualsEpsilon}, \ref{th:exampleKEquals3F2+epsilon} and Proposition \ref{th:exampleMoreGeneral} are not biharmonic. We begin with the family explored in Theorem \ref{th:exampleYZero}.
	
	\begin{theorem} \label{th:biharmonicityYEqualsZero}
		Let $\varphi: M^2 \to N^4(\epsilon)$ be a non-CMC W-surface with flat normal bundle. Assume that $3 f^2 + K - \epsilon \neq 0$ and $\nabla^\perp E_3 \neq 0$ at any point. If $M$ satisfies $A_3 (\grad f) = 2 f \grad f$, then it cannot be biharmonic.
	\end{theorem}

	\begin{proof}
		Assume that $M$ is biharmonic, thus it is biconservative. Eventually by restricting $M$, we can assume that $A_4 \neq 0$ and $\nabla _{E_2} E_2 \neq 0$ at any point. Locally, the system \eqref{eq:differentialSystemBiharmonicity} holds.
		
		We have seen that there exist local coordinates $(f, t)$ such that the functions $\uu$, $\vv$, $\ww$, $\xx$ and $\yy$ are given by \eqref{eq:vInTheCaseYEqualsZeroInTheChartFT} and \eqref{eq:solutionInTheCaseYEqualsZeroInTheChartFT}. From \eqref{eq:firstOrderDifferentialEquationMeanCurvatureCaseYZero} we obtain that
		\begin{equation*}
			\frac \partial {\partial s} = \frac 2 3 f \sqrt {C f^3 - c^{-6} f^6 - \epsilon} \frac \partial {\partial f}.
		\end{equation*}
		Replacing \eqref{eq:vInTheCaseYEqualsZeroInTheChartFT} and \eqref{eq:solutionInTheCaseYEqualsZeroInTheChartFT} in the sixth equation of \eqref{eq:differentialSystemBiharmonicity}, we obtain 
		\begin{equation*}
			5 c^{-9} f^{10} - 2 c^{-3} C f^7 -10 c^{-3} \epsilon f^4 + 18 c^3 C f^3 - 18 c^3 \epsilon = 0.
		\end{equation*}
		Since $c > 0$, we deduce that $f$ has to be a root of a non-zero polynomial with constant coefficients, so $f$ is constant, which is a contradiction.
	\end{proof}	

	Next, we analyze the family presented in Theorem \ref{th:exampleKEqualsEpsilon}.

	\begin{theorem} \label{th:biharmonicityKEqaulsEpsilon}
		Let $\varphi: M^2 \to N^4(\epsilon)$ be a non-CMC W-surface with flat normal bundle. Assume that $3 f^2 + K - \epsilon \neq 0$ and $\nabla^\perp E_3 \neq 0$ at any point. If $M$ has constant Gaussian curvature $K = \epsilon$, then it cannot be biharmonic.
	\end{theorem}

	\begin{proof}
		Assume that $M$ is biharmonic, thus it is biconservative. Eventually by restricting $M$, we can assume that $A_4 \neq 0$ and $\nabla _{E_2} E_2 \neq 0$ at any point. Locally, the system \eqref{eq:differentialSystemBiharmonicity} holds.
		
		We have seen that there exist local coordinates $(f, t)$ such that the functions $\uu$, $\vv$, $\ww$, $\xx$ and $\yy$ are given by \eqref{eq:vInTheCaseKEqualsEpsInTheChartFT} and \eqref{eq:solutionInTheCaseKEqualsEpsInTheChartFT}. From \eqref{eq:firstOrderODEForF} we obtain that 
		\begin{equation*}
			\frac \partial {\partial s} = \frac 1 3 f \sqrt {9 C f^3 - 4 \epsilon} \frac \partial {\partial f}.
		\end{equation*}
		Replacing \eqref{eq:vInTheCaseKEqualsEpsInTheChartFT} and \eqref{eq:solutionInTheCaseKEqualsEpsInTheChartFT} in the sixth equation of \eqref{eq:differentialSystemBiharmonicity}, we obtain that
		\begin{equation*}
			-18 c^4 f^6 + 18 c^2 \left ( 4 c + C \right ) f^5 - 36 c \left ( 3 c + 2 C \right ) f^4 + 9 ( 8 c + 9 C) f^3 + 16 c^2 \epsilon f^2  - 16 c \epsilon f - 36 \epsilon = 0.
		\end{equation*}
		Since $c > 0$, we deduce that $f$ has to be a root of a non-zero polynomial with constant coefficients, contradiction.
	\end{proof}
	 
	The following result shows that the biconservative surfaces presented in Proposition \ref{th:exampleMoreGeneral} cannot be biharmonic.
	
	\begin{theorem} \label{th:biharmonicityGeneralCase}
		Let $\varphi: M^2 \to N^4(\epsilon)$ be a non-CMC W-surface with flat normal bundle. Assume that $3 f^2 + K - \epsilon \neq 0$ and $\nabla^\perp E_3 \neq 0$ at any point. If $M$ satisfies
		\begin{equation*}
			\nabla _{E_2} E_2 = \frac 3 2 \frac {\dot f} f E_1,
		\end{equation*}
		then it cannot be biharmonic.
	\end{theorem}
	
	The proof of Theorem \ref{th:biharmonicityGeneralCase} is similar to the proofs of Theorems \ref{th:biharmonicityYEqualsZero} and \ref{th:biharmonicityKEqaulsEpsilon}.
	
	As in the previous cases, the biconservative surfaces presented in Theorem \ref{th:exampleKEquals3F2+epsilon} are not biharmonic.
	
	\begin{theorem} \label{th:biharmonicityKEquals3F2+epsilon}
		Let $\varphi: M^2 \to N^4(\epsilon)$ be a non-CMC W-surface with flat normal bundle. Assume that $\left \langle \nabla ^\perp _{E_1} E_3, E_4 \right \rangle \left \langle \nabla ^\perp _{E_2} E_3, E_4 \right \rangle = 0$ on $M$ and $\nabla^\perp E_3 \neq 0$ at any point. If $M$ satisfies $3 f^2 + K - \epsilon = 0$, then it cannot be biharmonic.
	\end{theorem}
	
	\section{Further developments}
	
	Inspired by Theorems \ref{th:biharmonicityYEqualsZero}, \ref{th:biharmonicityKEqaulsEpsilon}, \ref{th:biharmonicityGeneralCase} and \ref{th:biharmonicityKEquals3F2+epsilon}, we formulate the following conjecture.
	
	\begin{conjecture} \label{th:OpenProblem}
		Let $\varphi: M^2 \to N^4(\epsilon)$ be a non-CMC W-surface with flat normal bundle. Assume that $3 f^2 + K - \epsilon \neq 0$ and $\nabla ^\perp E_3 \neq 0$ at any point. Then, $M$ cannot be biharmonic.
	\end{conjecture}
	
	If Conjecture \ref{th:OpenProblem} proves to be true, then we obtain the classification of biharmonic W-surfaces with flat normal bundle in $N^4 (\epsilon)$. More precisely, we would have 
	\begin{theorem} [true under validity of Conjecture \ref{th:OpenProblem}] \label{th:theoremAfterOpenProblem} 
		Let $\varphi : M^2 \to N^4 (\epsilon)$ be a proper biharmonic W-surface with flat normal bundle. Assume that $3 f^2 + K - \epsilon \neq 0$ at any point. Then, $\epsilon > 0$, that is $N^4 (\epsilon)$ is the $4$-dimensional sphere $\mathbb S^4 (\epsilon)$, and the image $\varphi (M)$ lies minimally in the small hypersphere $\mathbb S^3 (2 \epsilon)$.
	\end{theorem}

	\begin{proof}
		First, suppose that $M$ is CMC. From a result in \cite{Oniciuc2002}, we obtain that $\epsilon > 0$ and, taking into account the main result of \cite{BalmusOniciuc2009}, we deduce that $\varphi(M)$ lies minimally in the small hypersphere $\mathbb S^3 (2 \epsilon)$.
		
		In the non-CMC case, it was proved in \cite{NistorOniciucTurgaySen2023}, \cite{NistorRusu2024} and \cite{SenTurgay2018} that there are no non-CMC, PNMC proper biharmonic surfaces in space forms.
		
		If Conjecture \ref{th:OpenProblem} proves to be true, then there are no non-CMC, non-PNMC proper biharmonic W-surfaces with flat normal bundle satisfying $3 f^2 + K - \epsilon \neq 0$ at any point immersed in $N^4 (\epsilon)$.
		
		We note that a minimal surfaces in $\mathbb S^3 (2 \epsilon)$ satisfies all conditions from the hypothesis.
	\end{proof}	
	
	Even if Theorem \ref{th:theoremAfterOpenProblem} will be an important result in the theory of biharmonic surfaces in $4$-dimensional space forms, and presumably hard to prove, it will represent just an intermediary step for a more general and difficult problem. In fact, the most important open problem for this topic is
	
	\begin{conjecture} \label{th:conjecture}
		Let $\varphi : M^2 \to N^4 (\epsilon)$ be a proper biharmonic immersion. Then, $\epsilon > 0$, that is $N^4 (\epsilon)$ is the $4$-dimensional sphere $\mathbb S^4 (\epsilon)$, and the image $\varphi (M)$ lies minimally in the small hypersphere $\mathbb S^3 (2 \epsilon)$.
	\end{conjecture}
	
	Taking into account a result in \cite{OniciucPHD}, the above statement can be rephrased as
	
	\noindent \textbf{Conjecture 2'.} Let $\varphi : M^2 \to N^4 (\epsilon)$ be a proper biharmonic immersion. Then, $\epsilon > 0$, that is $N^4 (\epsilon)$ is the $4$-dimensional sphere $\mathbb S^4 (\epsilon)$, and $|H| = \sqrt \epsilon$.

	\bibliographystyle{abbrv}
	\bibliography{Bibliography.bib}
\end{document}